\documentclass[11pt,letterpaper]{article}
\usepackage[margin=1in]{geometry}
\usepackage[T1]{fontenc}
\usepackage{lmodern,microtype,amsmath,amssymb,amsthm,mathtools,tikz}
\usepackage{fancyhdr,enumitem}
\setlist[enumerate]{itemsep=3pt,topsep=5pt}
\usepackage[colorlinks=true,linkcolor=blue!45!black,citecolor=blue!45!black,urlcolor=blue!45!black]{hyperref}
\usetikzlibrary{arrows.meta,patterns,calc}
\hypersetup{pdftitle={The Eigenvalue Region of a Family of Markov Chains with Randomness Confined to Two States},pdfauthor={Brecht Verbeken and Vincent Ginis},pdfsubject={Spectral regions, sparse realization, and Farey boundary selection}}
\newtheorem{theorem}{Theorem}[section]
\newtheorem{proposition}[theorem]{Proposition}
\newtheorem{lemma}[theorem]{Lemma}
\newtheorem{corollary}[theorem]{Corollary}
\numberwithin{equation}{section}
\DeclareMathOperator{\conv}{conv}
\DeclareMathOperator{\spec}{spec}

\DeclareMathOperator{\Arg}{Arg}
\renewcommand{\Re}{\operatorname{Re}}
\renewcommand{\Im}{\operatorname{Im}}
\newcommand{\C}{\mathbb C}
\newcommand{\interior}{\operatorname{int}}
\newcommand{\D}{\mathbb D}
\newcommand{\T}{\mathbb T}
\newcommand{\R}{\mathbb R}
\newcommand{\Z}{\mathbb Z}
\newcommand{\F}{\mathcal F}
\newcommand{\Sq}{\mathcal S_q}

\newcommand{\eps}{\varepsilon}
\title{The Eigenvalue Region of a Family of Markov Chains\\with Randomness Confined to Two States}
\author{%
  Brecht Verbeken\textsuperscript{1,2} and
  Vincent Ginis\textsuperscript{1,2,3}\\[0.8em]
  \parbox{0.96\textwidth}{\centering\small
    \textsuperscript{1}Department of Business Technology and Operations,\\
    Data Analytics Laboratory, Vrije Universiteit Brussel (VUB),\\
    Pleinlaan 2, 1050 Brussels, Belgium.\\[0.45em]
    \textsuperscript{2}imec-SMIT, Vrije Universiteit Brussel,\\
    Pleinlaan 9, 1050 Brussels, Belgium.\\[0.45em]
    \textsuperscript{3}School of Engineering and Applied Sciences,\\
    Harvard University, Cambridge, Massachusetts 02138, USA.\\[0.6em]
    \href{mailto:brecht.verbeken@vub.be}{brecht.verbeken@vub.be};
    \href{mailto:vincent.ginis@vub.be}{vincent.ginis@vub.be}
  }%
}
\date{}
\begin{document}
\maketitle
\begin{abstract}
We determine the complete eigenvalue region of the following Markov chains
on $2q$ states, for $q\ge2$.
The states are $A_0,\ldots,A_{q-1}$ and $B_0,\ldots,B_{q-1}$. Each transition
$A_j\to A_{j+1}$ and $B_j\to B_{j+1}$, for $0\le j<q-1$, has probability one.
From $A_{q-1}$, the chain moves to $A_0$ with probability $a$ and to $B_0$
with probability $1-a$. From $B_{q-1}$, it moves to $B_0$ with probability
$b$ and to $A_j$ with probability $(1-b)p_j$. Here $a,b\in[0,1]$ and
$p_j\ge0$ with $\sum_{j=0}^{q-1}p_j=1$; these parameters vary over all
permitted values.

Every segment from zero to an attainable eigenvalue lies in the region.
Its unit-circle points are exactly the roots of unity of order at most
$2q$, as for unrestricted stochastic matrices of the same order, yet the
full region is strictly smaller. It excludes
$0<|\Arg\lambda|<\pi/q$.
For odd $q$, the final nonreal boundary arc ends at a negative real point
inside the unit disk; the remaining interval to $-1$ is also boundary.

Every transition matrix has the same characteristic polynomial as one
with $a=b$. Among $q+1$ fixed matrices in this balanced family, convex
combinations of at most three realize every attainable eigenvalue, and two
suffice on the boundary. Both bounds are sharp within this set.
Consecutive reduced fractions in $[1/(2q),1/2]$ with denominators at most
$2q$ select the upper nonreal boundary constructions. Each boundary
radius is the unique solution of an equation whose left-hand side
increases strictly with the radius.
\end{abstract}

\section{Introduction}\label{sec:introduction}
Prescribing the transitions of a Markov chain raises a spectral question:
which eigenvalues remain possible when most states have a fixed successor?
For comparison, let $\Theta_n$ be the set of complex numbers that occur as
eigenvalues of $n\times n$ row-stochastic matrices, whose entries are
nonnegative and whose rows sum to one. Karpelevi\v{c}'s theorem determines
$\Theta_n$. In Ito's formulation, its boundary arcs are described by
polynomial equations associated with Farey neighbors: consecutive reduced
fractions in $[0,1]$ with denominators at most $n$
\cite{karpelevich,ito}. We impose the transition rules stated in the
abstract and defined in Section~\ref{sec:model}, with randomness confined
to two states, and determine the resulting region.

For $q\ge2$, the family retains all the unit-circle eigenvalues of unrestricted
stochastic matrices of order $2q$: the roots of unity of order at most
$2q$. Its full eigenvalue region is nevertheless strictly smaller.
It excludes the nonreal sector $0<|\Arg\lambda|<\pi/q$, and for odd $q$
its final nonreal boundary arc ends inside the unit disk rather than at
$-1$. Some unrestricted boundary arcs are retained exactly.
Section~\ref{sec:ambient} proves agreement between consecutive Farey
fractions when both denominators are at least $q$; this is a sufficient
condition, not a complete classification of agreement. These comparisons
show where the transition restrictions become visible even though the
unit-circle data are unchanged.

A close precedent is Kirkland's sharp eigenvalue region for stochastic
Leslie matrices \cite{leslie}. In their Markov-chain form, the chain
advances through ordered states with probability one until a single state
chooses the next destination. Our model lets two such choices interact.
The states form two sequences of $q$ states each. At the end of the first,
the chain chooses between the beginnings of the two sequences. At the end
of the second, it chooses between the beginning of that sequence and any
state of the first. The equal numbers of states and these different return rules give
two matching factors $X^q-a$ and $X^q-b$ in the characteristic polynomial.
They permit an exact reduction to equal probabilities $a=b$, as explained
after \eqref{eq:union}. The results concern this specified architecture;
they do not assert a classification for every chain with two rows of
transition probabilities free to vary.

We determine the whole eigenvalue region, including every radius below
the boundary, and give explicit matrices realizing its points.
The reduction to $a=b$ preserves the entire characteristic polynomial.
A further construction realizes any specified eigenvalue by a convex
combination of three matrices from a fixed set of $q+1$, with two
sufficing on the boundary. These bounds are sharp within that fixed set.
Thus the number of matrices needed for an individual eigenvalue remains
bounded as the number of states grows, even though the balanced family
has dimension $q$.

To identify the boundary, realizing roots of a polynomial is not enough:
some of those roots may lie inside the region. Johnson and Paparella
\cite{johnson} construct stochastic matrix families realizing the roots of
Ito's boundary polynomials. After cancellation of a common power of the
variable, these are called reduced Ito polynomials, with forms labelled
Types~0,~I,~II and~III. Kirkland and \v{S}migoc \cite{kirkland} characterize
all stochastic matrices with the Type~0 and~I characteristic polynomials,
and those with the fewest nonzero entries for Types~II and~III.
Munger, Nickerson and Paparella \cite{munger} separate realization,
continuity between endpoint angles, and actual boundary membership; their
analysis treats Type~I arcs and related arcs obtained by integer powers.
Kirkland, Laffey and \v{S}migoc \cite{revisited} identify the roots on the
unrestricted boundary and its radius at each angle. In our family, a
geometric eigenvalue criterion selects both the pair of fixed matrices
and the roots that give the boundary. This selection is the main task of
the proof.

Other constraints studied include a prescribed stationary distribution
\cite{stationary}, monotonicity \cite{monotone}, and a directed-cycle
transition pattern \cite{cycle}. For monotone matrices, after one step the
probability of being at or above a fixed state threshold is nondecreasing
with the initial state. In the directed-cycle family, each state can
remain in place or advance to the next state, with positive probability
of advancing. Our criterion instead uses the convex hull of
$1,\lambda,\ldots,\lambda^{q-1}$, whose weighted averages encode the
choices of destination from $B_{q-1}$.

\subsection{The transition rules}\label{sec:model}
Fix $q\ge2$, and label the states $A_0,\ldots,A_{q-1}$ and
$B_0,\ldots,B_{q-1}$. At every state except $A_{q-1}$ and $B_{q-1}$,
the next state is fixed:
\[
 A_j\longrightarrow A_{j+1},\qquad
 B_j\longrightarrow B_{j+1}\qquad(0\le j<q-1),
\]
and each of these transitions has probability one. To define the choices
at the two remaining states, choose $a,b\in[0,1]$ and a probability vector
$p=(p_0,\ldots,p_{q-1})$. We write $\Delta_{q-1}$ for the set of such vectors:
\[
 \Delta_{q-1}=\left\{p\in\R^q:p_j\ge0\ (0\le j<q),\quad
 \sum_{j=0}^{q-1}p_j=1\right\}.
\]
The probabilities at the two remaining states are
\[
 \begin{array}{c|c}
 \text{Transition}&\text{Probability}\\\hline
 A_{q-1}\longrightarrow A_0&a\\
 A_{q-1}\longrightarrow B_0&1-a\\
 B_{q-1}\longrightarrow B_0&b\\
 B_{q-1}\longrightarrow A_j&(1-b)p_j\quad(0\le j<q).
 \end{array}
\]
All unlisted transitions have probability zero, and probabilities in the
table may also be zero.

Figure~\ref{fig:chain} shows these rules for $q=4$.

\begin{figure}[htbp]
\centering
\begin{tikzpicture}[
 x=1cm,y=1cm,font=\small,>=Stealth,
 state/.style={circle,draw,minimum size=7mm,inner sep=0pt},
 choice state/.style={state,fill=black!10,thick},
 fixed transition/.style={->,semithick},
 choice/.style={->,thick,draw=blue!60!black},
 probability/.style={fill=white,inner sep=1.5pt,font=\small}
]
\foreach \panel/\offset in {L/0,R/7.9}{
 \begin{scope}[xshift=\offset cm]
  \foreach \j in {0,1,2}{
   \node[state] (\panel A\j) at (1.9*\j,2.8) {$A_{\j}$};
   \node[state] (\panel B\j) at (1.9*\j,0) {$B_{\j}$};
  }
  \node[choice state] (\panel A3) at (5.7,2.8) {$A_3$};
  \node[choice state] (\panel B3) at (5.7,0) {$B_3$};
  \foreach \j/\k in {0/1,1/2,2/3}{
   \draw[fixed transition] (\panel A\j) -- node[above] {$1$} (\panel A\k);
   \draw[fixed transition] (\panel B\j) -- node[above] {$1$} (\panel B\k);
  }
 \end{scope}
}
\draw[choice,rounded corners=10pt]
 (LA3.north) -- (5.7,4.0) -- (0,4.0) -- (LA0.north);
\node[above] at (2.85,4.0) {$a$};
\draw[choice] (LA3.south) -- node[probability,pos=.5] {$1-a$} (LB0.north);
\begin{scope}[xshift=7.9cm]
 \draw[choice,rounded corners=10pt]
  (RB3.south) -- (5.7,-1.05) -- (0,-1.05) -- (RB0.south);
 \node[below] at (2.85,-1.05) {$b$};
\end{scope}
\foreach \j in {0,1,2,3}{
 \draw[choice] (RB3.north) --
  node[probability,pos=.83] {$(1-b)p_{\j}$} (RA\j.south);
}
\node at (2.85,4.65) {(a) Choices at $A_3$};
\node at (10.75,4.65) {(b) Choices at $B_3$};
\end{tikzpicture}
\caption{The transition rules, shown for $q=4$. Arrow labels give the
probabilities. Both panels include the probability-one transitions; the
choices at $A_3$ and $B_3$ are shown separately to keep the arrows readable.
Combining the panels gives one eight-state chain. For general $q$,
$B_{q-1}$ can move to every $A_j$ with probability $(1-b)p_j$.}
\label{fig:chain}
\end{figure}

Order the states as $A_0,\ldots,A_{q-1},B_0,\ldots,B_{q-1}$.
Let $N_q$ be the $q\times q$ matrix with $(N_q)_{j,j+1}=1$ for
$0\le j<q-1$ and all other entries zero. For $0\le j<q$, let
$e_j\in\R^q$ be the column vector with entry $1$ at index $j$ and $0$
at every other index. In this ordering, the transition matrix is
\begin{equation}
 B_{2q}(a,b,p)=
 \begin{pmatrix}
 N_q+ae_{q-1}e_0^{\mathsf T}&(1-a)e_{q-1}e_0^{\mathsf T}\\
 (1-b)e_{q-1}p^{\mathsf T}&N_q+be_{q-1}e_0^{\mathsf T}
 \end{pmatrix}.\label{eq:model}
\end{equation}
Writing $\spec B$ for the set of eigenvalues of a matrix $B$, we define
\begin{equation}
 \Sigma_q=\bigcup_{a,b\in[0,1],\ p\in\Delta_{q-1}}
 \spec B_{2q}(a,b,p).\label{eq:union}
\end{equation}
Thus $\lambda\in\Sigma_q$ means that at least one permitted choice of
$a,b,p$ gives a matrix with eigenvalue $\lambda$. Different points of
$\Sigma_q$ may be realized by different matrices.

Throughout the paper, $\D=\{z\in\C:|z|\le1\}$ denotes the closed unit
disk, and $\T=\partial\D$ its unit circle. For $z\ne0$, $\Arg z$ is its
angle chosen in $(-\pi,\pi]$. Interiors and boundaries are taken in the
complex plane. For distinct $z,w\in\C$, $[z,w]$ is the closed line
segment joining them, and $(z,w)$ is that segment with both endpoints removed.

The two choices interact through the characteristic polynomial
computed in Proposition~\ref{prop:balancing}:
\[
 (X^q-a)(X^q-b)-(1-a)(1-b)\sum_{j=0}^{q-1}p_jX^j.
\]
The equal sequence lengths give the matching powers $X^q$ in its first
two factors. The return from $A_{q-1}$ enters the second sequence at
$B_0$, so the destination probabilities $p_j$ appear in the single factor
$\sum_jp_jX^j$ multiplying the interaction term $(1-a)(1-b)$.
Replacing $a,b$ by $c=(a+b)/2$ combines the first two factors into
$(X^q-c)^2$; replacing $p$ by a suitable weighted average of $p$ and $e_0$
preserves the original polynomial, including every root multiplicity.
This explains the role of the equal lengths and different return rules
in the reduction.

\subsection{Main conclusions and the geometric method}\label{sec:overview}
Call a matrix \emph{balanced} when $a=b$.
A \emph{balanced realization} of $\lambda$ is a permitted matrix
$B_{2q}(c,c,p)$ that has $\lambda$ as an eigenvalue; $c$ is the shared
value of $a$ and $b$. In such a matrix, the transitions from $A_{q-1}$
to $A_0$ and from $B_{q-1}$ to $B_0$ both have probability $c$.
 Every balanced matrix is a
weighted average of the following $q+1$ fixed matrices:
\begin{equation}
 U_q=B_{2q}(1,1,e_0),\qquad
 U_L=B_{2q}(0,0,e_{2q-L})\quad(q<L\le2q).\label{eq:vertices}
\end{equation}
In $U_q$, the $A$-states and the $B$-states form two separate cycles of
length $q$. In $U_L$, $A_{q-1}$ moves to $B_0$ and $B_{q-1}$ moves to
$A_{2q-L}$, both with probability one. This creates a cycle of length $L$;
any remaining $A$-states lead into that cycle. The indices therefore record
the cycle lengths.

\paragraph{Main conclusions.}
For every $q\ge2$, the following statements hold.
\begin{enumerate}
\item The region $\Sigma_q$ contains $[-1,1]$ and every segment from $0$
to an attainable eigenvalue. Its unit-circle points are exactly the
roots of unity of order at most $2q$, as for $\Theta_{2q}$, but
$\Sigma_q\subsetneq\Theta_{2q}$. There are no nonreal eigenvalues with
$0<|\Arg\lambda|<\pi/q$. On each ray with angle
$\theta\in[\pi/q,\pi)$, there is a largest attainable radius, and every
smaller radius is attainable.
The upper boundary curve approaches $-1$ for even $q$. For odd $q$, it
approaches $-\rho_q$ for some $0<\rho_q<1$, and $[-1,-\rho_q]$ is also
part of the boundary. The lower half of the region is its reflection
across the real axis.
\item Every transition matrix has the same characteristic polynomial as
a balanced matrix. For each specified $\lambda\in\Sigma_q$, there is a
realizing matrix
\[
 w_qU_q+w_MU_M+w_NU_N,\qquad q<M<N\le2q,
\]
where the coefficients are nonnegative and sum to one. For each boundary
point, a convex combination of two of $U_q,\ldots,U_{2q}$ suffices.
Coefficients may be zero, so these statements include individual fixed
matrices. The bounds of three and two are sharp within this fixed set,
for every $q\ge2$.
\end{enumerate}
Theorem~\ref{thm:main} gives the exact radius and boundary formulas;
Corollary~\ref{cor:sharpness} proves the sharpness of the realization
bounds. Figure~\ref{fig:region} compares $\Sigma_q$ with $\Theta_{2q}$
for $q=4$ and $q=5$.

\begin{figure}[htbp]
\centering
\begin{minipage}{.49\linewidth}\centering

\caption{The restricted regions $\Sigma_4$ (eight states, left) and
$\Sigma_5$ (ten states, right), shaded gray, compared with the blue
boundaries of the unrestricted regions $\Theta_8$ and $\Theta_{10}$.
Black lines trace $\partial\Sigma_q$; where a black curve lies on a blue
curve, the two boundaries agree. The blue curves use the radius
characterization in \cite{revisited}. Both regions have the same
unit-circle points, marked by dots on the dotted circle, but only the
unrestricted region contains nonreal points in the sector
$0<|\Arg\lambda|<\pi/q$. For $q=4$, both nonreal boundary curves approach
$-1$. For $q=5$, the restricted curves approach
$-\rho_5=-2^{-1/3}$ inside the disk; the thick black interval
$[-1,-\rho_5]$ is also boundary. It and the final restricted arcs are
realized by convex combinations of $U_9$ and $U_6$.}
\label{fig:region}
\end{figure}

The set of balanced matrices is
\begin{equation}
 \begin{gathered}
 \Sq=\conv\{U_q,U_{q+1},\ldots,U_{2q}\},\\
 B_{2q}(c,c,p)=cU_q+(1-c)\sum_{j=0}^{q-1}p_jU_{2q-j}.
 \end{gathered}\label{eq:simplex}
\end{equation}
Here $\conv$ denotes the convex hull: the set of weighted averages with
nonnegative weights summing to one. The weights in this representation
are uniquely determined by the matrix, so $\Sq$ is a $q$-dimensional
simplex with vertices $U_q,\ldots,U_{2q}$. An edge consists of the weighted
averages of two vertices. A \emph{triangular face} consists of all weighted
averages of three distinct vertices, with every other vertex assigned
weight zero. In particular, for $q<M<N\le2q$,
\[
 \conv\{U_q,U_M,U_N\}
 =\left\{w_qU_q+w_MU_M+w_NU_N:
 w_q,w_M,w_N\ge0,\quad w_q+w_M+w_N=1\right\}.
\]
The weights may be zero, so the triangular face includes its three edges
and vertices. For $q=2$, this triangle is the whole simplex $\mathcal S_2$. We call $U_q$ the \emph{apex}, an edge
$[U_M,U_N]$ with $M,N>q$ a \emph{base edge}, and an edge $[U_q,U_L]$
a \emph{lateral edge}.

The reduction to $a=b$ preserves the entire characteristic polynomial,
depends continuously on the original matrix, and leaves every balanced
matrix unchanged. The geometric construction below then gives each
specified $\lambda\in\Sigma_q$ a balanced realization in which at most
two entries of $p$ are positive. If the only possibly nonzero entries are
$p_j$ and $p_k$, with $0\le j<k<q$, then the formula above becomes
\[
 B_{2q}(c,c,p)=cU_q+(1-c)p_jU_{2q-j}+(1-c)p_kU_{2q-k}.
\]
The three coefficients are nonnegative and sum to one. Thus this matrix
lies in the triangle whose vertices are $U_q$, $U_{2q-j}$ and $U_{2q-k}$.
Writing $M=2q-k$ and $N=2q-j$ gives $q<M<N\le2q$, so this triangle is
the face $\conv\{U_q,U_M,U_N\}$ of $\Sq$. If $\lambda\in\partial\Sigma_q$,
at most two vertices suffice. These further reductions may change the
other eigenvalues. Both bounds are sharp: some attainable eigenvalues
require three vertices, and some boundary points require two
(Corollary~\ref{cor:sharpness}).

For $\lambda\in\D$ with $\lambda^q\ne1$, the eigenvalue criterion is
\[
 \lambda\in\Sigma_q\quad\Longleftrightarrow\quad
 \bigl(1+s(\lambda^q-1)\bigr)^2
 \in\conv\{1,\lambda,\ldots,\lambda^{q-1}\}
 \quad\text{for some }s\ge1.
\]
Here $s=(1-c)^{-1}$, where $c=a=b\in[0,1)$. The left-hand side of the
inclusion traces a quadratic curve as $s$ varies. The convex hull on the
right consists of the averages $\sum_jp_j\lambda^j$ allowed by the
probability vector $p$. The exceptional case $\lambda^q=1$ is already
realized by $U_q$.

An arbitrary point in this convex hull may require three of the listed
complex numbers in a weighted average. The last intersection of the
quadratic curve with the hull lies on its boundary, where two suffice.
Using their coefficients as the two nonzero entries of $p$ gives the
three-matrix realization above. Theorem~\ref{thm:triangles} proves this
construction, which preserves the specified eigenvalue but may change
the others.

For $q\ge3$ and a boundary point with $|\lambda|<1$ and
$\pi/q<\Arg\lambda<\pi$, Proposition~\ref{prop:alternatives} gives
two forms for the last intersection. If it occurs at $s=1$, it lies on an
edge of the convex hull. Then $c=0$, giving a realization on a base edge
of $\Sq$. If it occurs at $s>1$, it is a vertex $\lambda^j$. Taking
$p_j=1$ gives a weighted average of $U_q$ and $U_{2q-j}$, on a lateral
edge of $\Sq$.

Dubuc and Malik \cite{dubuc} relate convex hulls generated by powers
of a complex number to polynomial equations with three nonzero terms and
to Farey fractions. For $q\ge3$, we classify the last intersections
between the quadratic curve and $\conv\{1,\lambda,\ldots,\lambda^{q-1}\}$
when $\lambda$ is a boundary point with $|\lambda|<1$ and
$\pi/q<\Arg\lambda<\pi$. At a vertex intersection $\lambda^j$, the squared
equation allows two possible square roots, $\xi$ and $-\xi$, with
$\xi^2=\lambda^j$. The boundary construction requires the root for which
the line through $\xi$ and $\lambda^j$ touches the convex hull and keeps
the hull and the origin on the same side. Squaring the equation loses
this geometric distinction.

For consecutive reduced fractions $f<g$ in $[1/(2q),1/2]$ with
denominators at most $2q$, the resulting boundary radius is unique at
every angle $2\pi x$ with $f<x<g$. The equation determining that radius
is derived from the geometric classification and involves powers of the
radius and sines of angles.

To determine the whole region, we must also prove that $t\lambda\in\Sigma_q$ whenever
$\lambda\in\Sigma_q$ and $0\le t\le1$: every segment from the origin to
an attainable eigenvalue lies in the region. This property is called
\emph{star-convexity with respect to the origin}. Vagenende, Verbeken and
Guerry \cite{starconvex} give constructive proofs for several classes of
stochastic matrices by averaging a given matrix with one whose rows are
all identical. A nontrivial average of this kind cannot preserve the
probability-one transitions in our model, because their destinations
differ from state to state.

For $q\ge3$, the proof combines three facts: the selected radius is the
largest attainable radius, every nonreal boundary eigenvalue inside the
unit disk can be approached by interior eigenvalues, and the region below
each selected boundary arc is connected and contains no boundary points.
The approximation supplies attainable points there; connectedness then
forces every point there to be attainable. Section~\ref{sec:filling}
gives the full argument. It also applies to repeated boundary eigenvalues.
The four-state case $q=2$ has a separate proof in
Appendix~\ref{app:four}.

\subsection{Farey fractions and their matrix edges}\label{sec:farey}
The Farey sequence $F_n$ consists of the reduced fractions in $[0,1]$
whose denominators are at most $n$, arranged in increasing order. We use
\begin{equation}
 \F_q=F_{2q}\cap\left[\frac1{2q},\frac12\right],\qquad
 \ell_q(k)=k\lceil q/k\rceil.\label{eq:farey}
\end{equation}
The number $\ell_q(k)$ is the smallest multiple of $k$ that is at least
$q$; it satisfies $q\le\ell_q(k)\le2q$ for $k\le2q$. It therefore gives
the cycle length of one of our fixed matrices. Since $k$ divides
$\ell_q(k)$, the complex number $e^{2\pi ih/k}$ is an
$\ell_q(k)$th root of unity and hence an eigenvalue of $U_{\ell_q(k)}$.
Thus each fraction $h/k$ specifies a unit-circle eigenvalue and a fixed
matrix realizing it.

For consecutive fractions $f=h/k<g=r/s$ in $\F_q$, set
\begin{equation}
 E_{f,g}=[U_M,U_N],\qquad M=\ell_q(k),\quad N=\ell_q(s).\label{eq:edge}
\end{equation}
This is the edge consisting of weighted averages of the two endpoint
matrices. We will show that these averages realize the boundary at
angles $2\pi x$ with $f<x<g$. For example, when $q=4$, the consecutive
fractions $1/5<1/4$ select the edge between $U_5$ and the apex $U_4$,
which is a lateral edge.

The denominators of consecutive fractions satisfy $k+s>2q$: otherwise
$(h+r)/(k+s)$ would be a fraction between them with an allowed denominator.
Consequently, $M$ and $N$ cannot both equal $q$. That would force both
$k$ and $s$ to divide $q$, giving $k+s\le2q$. Thus the selected edge is a
base edge when both lengths exceed $q$, and a lateral edge when one
length equals $q$.

\subsection{The boundary radius and the complete region}\label{sec:radius}
Within each interval of angles selected by consecutive fractions, we
compute the boundary radius by solving an equation in a single real
variable $\rho$. For exponents $u,v>0$ and angles $\alpha,\beta>0$ with
$\alpha+\beta<\pi$, define $R_{u,v}(\alpha,\beta)$ to be the solution
$\rho\in(0,1)$ of
\begin{equation}
 \rho^v\sin\alpha+\rho^u\sin\beta=\sin(\alpha+\beta).\label{eq:scalar}
\end{equation}
The left-hand side is continuous and strictly increasing on $[0,1]$.
It is below the right-hand side at $\rho=0$ and above it at $\rho=1$,
so there is exactly one solution between them. Lemma~\ref{lem:chord}
also gives a geometric interpretation: at this radius, the line segment
joining $\rho^{-u}e^{i\alpha}$ and $\rho^{-v}e^{-i\beta}$ contains $1$
strictly between its endpoints.

Fix consecutive fractions $f=h/k<g=r/s$ in $\F_q$, and choose
$x\in(f,g)$. Write a candidate eigenvalue in polar form as
$\lambda=\rho e^{2\pi ix}$: its distance from the origin is $\rho$, and
its angle is $2\pi x$. We denote the boundary radius at this angle by
$r_q(x)$. Table~\ref{tab:boundary-constructions} summarizes the two
constructions; the formulas below define their angles and coefficients.

\paragraph{Base edges.}
The \emph{base} of $\Sq$ is
$\conv\{U_{q+1},\ldots,U_{2q}\}$: the set of weighted averages of these
vertices, with no contribution from $U_q$. When $M,N>q$, the edge
$[U_M,U_N]$ lies in this base. We now find a matrix on that edge with
an eigenvalue at the boundary radius for angle $2\pi x$.

Since $M$ is a multiple of $k$ and $N$ a multiple of $s$, write $M=uk$
and $N=vs$ with positive integers $u,v$. Define
\begin{equation}
 \alpha=2\pi u(kx-h),\qquad \beta=2\pi v(r-sx),\qquad
 r_q(x)=R_{M,N}(\alpha,\beta).\label{eq:base-data}
\end{equation}
Raising $\lambda=\rho e^{2\pi ix}$ to the powers $M$ and $N$ multiplies
its angle by $M$ and $N$, respectively. Subtracting whole turns gives
\[
 \lambda^M=\rho^M e^{i\alpha},\qquad
 \lambda^N=\rho^N e^{-i\beta}.
\]
Thus $\lambda^M$ lies at angle $\alpha$ above the positive real axis,
and $\lambda^N$ at angle $\beta$ below it. For $x\in(f,g)$ these angles
satisfy $\alpha,\beta>0$ and $\alpha+\beta<\pi$, as proved below, so the
equation \eqref{eq:scalar} determines a unique radius. At $\rho=r_q(x)$, set
\begin{equation}
 t=\frac{\rho^N\sin\alpha}{\sin(\alpha+\beta)}.\label{eq:base-t}
\end{equation}
Equation~\eqref{eq:scalar} gives
$1-t=\rho^M\sin\beta/\sin(\alpha+\beta)>0$. Since $t>0$ as well,
$0<t<1$: the coefficients $1-t$ and $t$ are positive and sum to one.
Consequently, $(1-t)U_M+tU_N$ is a weighted average of the two matrices.
Each transition probability is $1-t$ times the corresponding probability
in $U_M$ plus $t$ times that in $U_N$. This matrix has
$\lambda=\rho e^{2\pi ix}$ as an eigenvalue, because $\lambda$ satisfies
its characteristic equation
\begin{equation}
 \lambda^{2q}=(1-t)\lambda^{2q-M}+t\lambda^{2q-N}.\label{eq:base-equation}
\end{equation}

\paragraph{Lateral edges.}
Suppose the edge joins $U_q$ to another vertex $U_L$. For this case,
denote the endpoint associated with $U_q$ by $h/k$, whether it is the
left or the right endpoint, and denote the other endpoint by $\ell/L$.
Then $k\mid q$ and $q<L<2q$. Set $m=qh/k\in\Z$ and $J=2q-L$;
in the chain defined by $U_L$, the transition from $B_{q-1}$ goes to
$A_J$. The sign $\eps$ below records which side of $h/k$ contains $x$:
\begin{equation}
 \begin{gathered}
 \eps=\begin{cases}1,&h/k<x<\ell/L,\\-1,&\ell/L<x<h/k,\end{cases}\\
 \alpha=2\pi\eps(qx-m),\qquad \beta=\pi\eps(\ell-Lx),\qquad
 r_q(x)=R_{q,L/2}(\alpha,\beta).
 \end{gathered}\label{eq:lateral-data}
\end{equation}
Again $\alpha,\beta>0$ and $\alpha+\beta<\pi$. Equation~\eqref{eq:scalar} now
uses exponents $q$ and $L/2$. At $\rho=r_q(x)$, choose the square root
of $\lambda^J$ given by
\begin{equation}
 \xi(\lambda)=(-1)^\ell\rho^{J/2}e^{\pi iJx},\qquad \xi(\lambda)^2=\lambda^J.\label{eq:root}
\end{equation}
The factor $(-1)^\ell$ selects between the two square roots. It gives
$\xi(\lambda)=\rho^{J/2}e^{i\eps(\alpha+\beta)}$, whereas
$\lambda^q=\rho^q e^{i\eps\alpha}$. Thus their angles are
$\eps(\alpha+\beta)$ and $\eps\alpha$, respectively. Using these
angles, equation~\eqref{eq:scalar} places $\lambda^q$ strictly inside
the line segment joining $1$ and $\xi(\lambda)$. Thus the coefficient
\begin{equation}
 c=\frac{\lambda^q-\xi(\lambda)}{1-\xi(\lambda)}\in(0,1),\qquad
 \lambda^q-c=(1-c)\xi(\lambda)\label{eq:lateral-equation}
\end{equation}
is real and gives the averaging weight in
$cU_q+(1-c)U_L=B_{2q}(c,c,e_J)$. This matrix has $\lambda$ as an
eigenvalue: squaring the second identity yields its characteristic
equation,
\[
 (\lambda^q-c)^2=(1-c)^2\lambda^J.
\]
The squared equation alone does not determine which square root to use
in the boundary construction; that choice is specified in \eqref{eq:root}.

\begin{table}[htbp]
\centering\small
\renewcommand{\arraystretch}{1.3}
\begin{tabular}{@{}p{.23\linewidth}p{.31\linewidth}p{.40\linewidth}@{}}
\hline
Selected matrices & Realizing matrix & Radius and coefficient\\
\hline
$U_M,U_N$, $M,N>q$\newline (base edge)
& $(1-t)U_M+tU_N$
& $r_q(x)=R_{M,N}(\alpha,\beta)$\newline
  \eqref{eq:base-data}, $t$ from \eqref{eq:base-t}\\
$U_q,U_L$, $q<L<2q$\newline (lateral edge)
& $cU_q+(1-c)U_L$
& $r_q(x)=R_{q,L/2}(\alpha,\beta)$\newline
  \eqref{eq:lateral-data}, $c$ from \eqref{eq:lateral-equation}\\
\hline
\end{tabular}
\caption{Boundary constructions for $x$ strictly between consecutive
fractions in $\F_q$. The endpoint fractions select the two matrices by
\eqref{eq:edge}. In both rows, the radius solves \eqref{eq:scalar}, and
the two matrix coefficients are positive and sum to one.}
\label{tab:boundary-constructions}
\end{table}

At every Farey fraction $x\in\F_q$ below $1/2$, including $1/(2q)$,
set $r_q(x)=1$: the boundary reaches the unit circle at these angles.
The limiting radius as $x$ approaches $1/2$, the angle of the negative
real axis, is $1$ for even $q$. For odd $q$, denote it by $\rho_q$;
it is the unique solution in $(0,1)$ of
\begin{equation}
 (2q-1)\rho_q^{q+1}+(q+1)\rho_q^{2q-1}=q-2.\label{eq:odd-limit}
\end{equation}
Uniqueness follows because the left-hand side increases strictly from
$0$ to $3q$ as $\rho_q$ ranges from $0$ to $1$, while $0<q-2<3q$.
\begin{theorem}[Spectral region and sparse realization]\label{thm:main}
Let $q\ge2$. The attainable eigenvalues in the upper half-plane are
\begin{equation}
 \Sigma_q\cap\{\Im\lambda\ge0\}
 =[-1,1]\cup\left\{\rho e^{2\pi ix}:\frac1{2q}\le x<\frac12,\quad0<\rho\le r_q(x)\right\}.\label{eq:region}
\end{equation}
Thus every real number in $[-1,1]$ is attainable. At each angle $2\pi x$
with $1/(2q)\le x<1/2$, every radius from $0$ to $r_q(x)$ is attainable;
there are no nonreal eigenvalues with angle strictly between $0$ and
$\pi/q$. The eigenvalues in the lower half-plane are the complex
conjugates of those in the upper half-plane.

On every interval between consecutive fractions in $\F_q$, the function
$r_q$ is real-analytic and is determined by \eqref{eq:scalar}, using the
data in \eqref{eq:base-data} or \eqref{eq:lateral-data}, respectively.
Its limit is $1$ at $1/(2q)$ and at every fraction $f\in\F_q$ with
$1/(2q)<f<1/2$. As $x\uparrow1/2$, its limit is $1$ for even $q$
and $\rho_q$ for odd $q$.

Define the curved part of the upper boundary by
$C_q^+=\{r_q(x)e^{2\pi ix}:1/(2q)<x<1/2\}$. Then
\begin{equation}
 \partial\Sigma_q\cap\{\Im\lambda\ge0\}
 =[0,1]\cup[0,e^{i\pi/q}]\cup C_q^+\cup E_q,\label{eq:boundary}
\end{equation}
where $[0,e^{i\pi/q}]$ denotes the straight line segment from the origin
to $e^{i\pi/q}$, and $E_q=\{-1\}$ for even $q$, while
$E_q=[-1,-\rho_q]$ for odd $q$. The eigenvalues on the unit circle are
exactly the roots of unity whose orders are at most $2q$:
\begin{equation}
 \Sigma_q\cap\T=\{\zeta:\zeta^m=1\text{ for some }1\le m\le2q\}.\label{eq:unit}
\end{equation}

Every $\lambda\in\Sigma_q$ is an eigenvalue of a matrix of the form
\[
 w_qU_q+w_MU_M+w_NU_N,\qquad q<M<N\le2q,
\]
with $w_q,w_M,w_N\ge0$ and $w_q+w_M+w_N=1$. Such matrices form the
triangular face $\conv\{U_q,U_M,U_N\}$ of $\Sq$: at most three vertices
are needed, one of which is $U_q$. Every $\lambda\in\partial\Sigma_q$
is an eigenvalue of a matrix of the form
\[
 (1-t)U_K+tU_L,\qquad q\le K<L\le2q,\quad0\le t\le1.
\]
This is a convex combination of two of the fixed matrices; $t=0$ or
$t=1$ gives an individual matrix. Thus a realizing matrix can be chosen
on an edge of $\Sq$.
\end{theorem}

The theorem shows that every line segment from the origin to an
attainable eigenvalue lies in $\Sigma_q$. There are no nonreal eigenvalues
in the sector $0<|\Arg\lambda|<\pi/q$ around the positive real axis.
For odd $q$, the final upper boundary arc approaches $-\rho_q$.
Both this arc and the boundary interval $[-1,-\rho_q]$ are realized by
weighted averages of $U_{2q-1}$ and $U_{q+1}$; see
Figure~\ref{fig:region}. Section~\ref{sec:collision} shows that, for one
matrix on this edge, $-\rho_q$ is an eigenvalue of multiplicity exactly
two: it is a double root of that matrix's characteristic polynomial.

Section~\ref{sec:transfer} proves that the reduction to $a=b$ preserves
the characteristic polynomial and that three vertices suffice to realize
any attainable eigenvalue. It also shows, for $q\ge3$, that nonreal
eigenvalues inside the unit disk can be approached by eigenvalues in the
interior of $\Sigma_q$. Section~\ref{sec:contacts} identifies the edges
that realize boundary eigenvalues and their associated Farey intervals.
Section~\ref{sec:filling} proves that no eigenvalue lies beyond the
computed boundary radius and that every smaller radius is attainable.
Section~\ref{sec:endpoints} treats the endpoint limits, shows that some
eigenvalues require three vertices and some boundary eigenvalues require
two, and gives explicit realizing matrices. Section~\ref{sec:ambient}
compares $\Sigma_q$ with the eigenvalue region of all $2q\times2q$
stochastic matrices. Section~\ref{sec:conclusion} returns to what the
classification shows about prescribed transitions and the scope of the
method. The appendices supply the arithmetic argument used
in the boundary classification and a separate proof for the four-state
case $q=2$.
The classification and attainment of smaller radii are proved directly
from our transition rules, independently of the directed-cycle
classification \cite{cycle} and the general proof of the
Karpelevi\v{c} theorem in \cite{structural}.

\section{Transfer geometry and interior approximation}\label{sec:transfer}
\subsection{Balancing the two probabilities}
The transition from $B_{q-1}$ to $A_j$ has probability $(1-b)p_j$.
To encode the probability vector $p=(p_0,\ldots,p_{q-1})$ in a polynomial,
write
\[
 P_p(X)=\sum_{j=0}^{q-1}p_jX^j.
\]

\begin{proposition}[Balancing preserves the characteristic polynomial]\label{prop:balancing}
The characteristic polynomial of $B_{2q}(a,b,p)$ is
\begin{equation}
 \chi_{a,b,p}(X)=(X^q-a)(X^q-b)-(1-a)(1-b)P_p(X).\label{eq:characteristic}
\end{equation}
To replace $a$ and $b$ by equal probabilities, set $c=(a+b)/2$.
If $c<1$, define
\begin{equation}
 \tau=\frac{(1-a)(1-b)}{(1-c)^2},\qquad \widetilde p=\tau p+(1-\tau)e_0.\label{eq:balance}
\end{equation}
Here $0\le\tau\le1$, so $\widetilde p$ is again a probability vector:
it combines the original vector $p$ with $e_0$, which assigns all
probability to $A_0$. The resulting balanced matrix satisfies
\[
 \chi_{a,b,p}(X)=\chi_{c,c,\widetilde p}(X).
\]
It therefore has the same eigenvalues, with the same algebraic
multiplicities, as the original matrix. If $c=1$, then $a=b=1$ and the
matrix is already $U_q$; take $U_q$ as its image.

This construction defines a continuous map from the full matrix family
onto the balanced family $\Sq$. Every matrix with $a=b$ is left
unchanged by the map.
\end{proposition}

\begin{proof}
Let $z$ be a right eigenvector with eigenvalue $\lambda$, and set
$u=z_{A_0}$ and $v=z_{B_0}$. The probability-one transitions give
$z_{A_{j+1}}=\lambda z_{A_j}$ and $z_{B_{j+1}}=\lambda z_{B_j}$ for
$j<q-1$. Thus
$z_{A_j}=\lambda^ju$ and $z_{B_j}=\lambda^jv$.
The eigenvalue equations at the two remaining states, $A_{q-1}$ and
$B_{q-1}$, become
\[
 \begin{pmatrix}\lambda^q-a&-(1-a)\\-(1-b)P_p(\lambda)&\lambda^q-b\end{pmatrix}
 \begin{pmatrix}u\\v\end{pmatrix}=0.
\]

To compute the characteristic polynomial, including root multiplicities,
let $D=XI_q-N_q$ for $X\ne0$. Since $N_q^q=0$,
$D^{-1}=\sum_{r=0}^{q-1}X^{-r-1}N_q^r$, and hence
\[
 e_0^{\mathsf T}D^{-1}e_{q-1}=X^{-q},\qquad
 p^{\mathsf T}D^{-1}e_{q-1}=X^{-q}P_p(X).
\]
The matrix $B_{2q}(a,b,p)$ differs from $\operatorname{diag}(N_q,N_q)$
only in the rows for $A_{q-1}$ and $B_{q-1}$. Factoring
$\operatorname{diag}(D,D)$ on the right of $XI-B_{2q}(a,b,p)$ leaves
a matrix that differs from the identity only in those two rows.
Expanding its determinant along the unchanged rows gives
\[
 \det(XI-B_{2q}(a,b,p))
 =X^{2q}\det\begin{pmatrix}
 1-aX^{-q}&-(1-a)X^{-q}\\
 -(1-b)X^{-q}P_p(X)&1-bX^{-q}
 \end{pmatrix}.
\]
Evaluating this determinant gives \eqref{eq:characteristic} for $X\ne0$.
Both sides are polynomials, so the equality also holds at $X=0$.

For $c<1$, the arithmetic--geometric mean inequality gives
\[
 0\le(1-a)(1-b)\le
 \left(\frac{(1-a)+(1-b)}2\right)^2=(1-c)^2,
\]
so $0\le\tau\le1$. The definitions of $c$ and $\tau$ also give
\[
 (1-c)^2\tau=(1-a)(1-b),\qquad
 (1-c)^2(1-\tau)=\frac{(a-b)^2}{4},
\]
and
$(X^q-c)^2-(X^q-a)(X^q-b)=(a-b)^2/4$.
Since $P_{\widetilde p}(X)=\tau P_p(X)+(1-\tau)$, substitution yields
\[
 \begin{aligned}
 \chi_{c,c,\widetilde p}(X)
 &=(X^q-c)^2-(1-c)^2\bigl[\tau P_p(X)+(1-\tau)\bigr]\\
 &=(X^q-a)(X^q-b)-(1-a)(1-b)P_p(X)
 =\chi_{a,b,p}(X).
 \end{aligned}
\]

In the weighted-average representation \eqref{eq:simplex}, the coefficient
of $U_q$ is the transition probability from $A_{q-1}$ to $A_0$, and the
coefficient of $U_{2q-j}$ is the transition probability from $B_{q-1}$
to $A_j$. These coefficients are therefore uniquely determined by the
matrix, proving that the vertices are affinely independent.

The balancing map is well defined on matrices even when their parameters
are not unique. For $c<1$, the identity
\[
 \tau p=\frac{1-a}{(1-c)^2}\bigl((1-b)p\bigr)
\]
shows that its image depends continuously on the original transition
probabilities: $(1-b)p$ records the transitions from $B_{q-1}$ to the
$A_j$. In particular, when $b=1>a$, the input matrix does not record $p$,
but $\tau=0$ and $\widetilde p=e_0$, so the image is independent of it.
As $c\to1$, the total weight $1-c$ assigned to vertices other than $U_q$
tends to zero, so the image tends to $U_q$. The map is thus continuous
on the whole family. Finally, if $a=b<1$, then $c=a$, $\tau=1$, and
$\widetilde p=p$; the case $a=b=1$ is fixed by definition. Every balanced
matrix is therefore left unchanged.
\end{proof}

In this model, the polynomial $P_p$ in the characteristic-polynomial
formula can be any polynomial of the form
\[
 P_p(X)=\sum_{j=0}^{q-1}p_jX^j,\qquad p_j\ge0,\quad\sum_jp_j=1.
\]
Each such polynomial encodes a permitted choice of return probabilities
$p$. The constant polynomial $1$ corresponds to $p=e_0$, which assigns
all return probability to $A_0$. For $c<1$, balancing replaces $P_p$ by
$\tau P_p+(1-\tau)$, with $0\le\tau\le1$, so the adjusted polynomial
is again of this form. More generally, the same algebraic argument
applies whenever the set of polynomials permitted in place of $P_p$
contains $1$ and is closed under weighted averages with nonnegative
weights summing to one. The geometric arguments below use the specific
probability polynomials displayed above.

\subsection{A geometric criterion for eigenvalues}
For a candidate eigenvalue $\lambda$, define
\begin{equation}
 K_q(\lambda)=\conv\{1,\lambda,\ldots,\lambda^{q-1}\},\qquad
 \Gamma_\lambda(s)=\bigl(1+s(\lambda^q-1)\bigr)^2.\label{eq:transfer}
\end{equation}
The set $K_q(\lambda)$ consists of all weighted averages of the listed
complex numbers. It is exactly the set of possible values of $P_p(\lambda)$
as the probability vector $p$ varies. For a balanced matrix with $c<1$,
dividing the eigenvalue equation by $(1-c)^2$ gives
\[
 \Gamma_\lambda(s)=P_p(\lambda),\qquad s=(1-c)^{-1}\ge1.
\]
Thus the eigenvalue question becomes a geometric one: does the curve
$\Gamma_\lambda(s)$ enter $K_q(\lambda)$ for some $s\ge1$?

If $\lambda\in\D$, then $K_q(\lambda)\subseteq\D$, since all the
listed powers have modulus at most one. Any intersection must therefore
satisfy $|1+s(\lambda^q-1)|\le1$. When $\lambda^q\ne1$ and $s\ge1$,
expanding this inequality gives the upper bound $s\le s^*(\lambda)$, where
\begin{equation}
 s^*(\lambda)=\frac{2(1-\Re\lambda^q)}{|1-\lambda^q|^2}.\label{eq:sstar}
\end{equation}
This upper bound is at least $1$, because
\[
 2(1-\Re\lambda^q)-|1-\lambda^q|^2=1-|\lambda|^{2q}\ge0.
\]
For these candidates, it is therefore enough to check the closed interval
$1\le s\le s^*(\lambda)$.

\begin{proposition}[Geometric eigenvalue criterion]\label{prop:transfer}
If $\lambda^q=1$, then $\lambda$ is an eigenvalue of $U_q$ and hence
belongs to $\Sigma_q$. Otherwise,
\begin{equation}
 \lambda\in\Sigma_q\quad\Longleftrightarrow\quad
 \lambda\in\D\ \text{ and }\ \Gamma_\lambda([1,s^*(\lambda)])\cap K_q(\lambda)\ne\varnothing.\label{eq:criterion}
\end{equation}
That is, $\lambda$ must lie in the unit disk, and the curve point
$\Gamma_\lambda(s)$ must belong to $K_q(\lambda)$ for at least one
$s\in[1,s^*(\lambda)]$.

To construct a realizing matrix from such an intersection, choose a
probability vector $p$ satisfying
\[
 \Gamma_\lambda(s)=\sum_{j=0}^{q-1}p_j\lambda^j,
\]
and set $c=1-1/s$. Then $B_{2q}(c,c,p)$ has $\lambda$ as an eigenvalue,
with the right eigenvector
\begin{equation}
 (1,\lambda,\ldots,\lambda^{q-1},w,\lambda w,\ldots,\lambda^{q-1}w)^{\mathsf T},\qquad
 w=1+s(\lambda^q-1).\label{eq:eigenvector}
\end{equation}
\end{proposition}

\begin{proof}
The case $\lambda^q=1$ is realized by $U_q$. Now suppose
$\lambda^q\ne1$. If $\lambda\in\Sigma_q$, balancing gives a realizing
matrix $B_{2q}(c,c,p)$. Necessarily $c<1$, since $c=1$ gives $U_q$,
whose eigenvalues satisfy $\lambda^q=1$. The characteristic-polynomial
formula gives
\[
 (\lambda^q-c)^2=(1-c)^2P_p(\lambda).
\]
With $s=(1-c)^{-1}\ge1$, division by $(1-c)^2$ yields
$\Gamma_\lambda(s)=P_p(\lambda)\in K_q(\lambda)$. Stochasticity gives
$\lambda\in\D$, and the upper bound derived above gives
$s\le s^*(\lambda)$. This proves the forward implication.

Conversely, suppose $\lambda\in\D$ and
$\Gamma_\lambda(s)\in K_q(\lambda)$ for some
$s\in[1,s^*(\lambda)]$. Membership in this convex hull means that
$\Gamma_\lambda(s)=\sum_jp_j\lambda^j$ for a probability vector $p$.
Set $c=1-1/s\in[0,1)$. Reversing the preceding substitution gives
$(\lambda^q-c)^2=(1-c)^2P_p(\lambda)$, so the characteristic polynomial
of $B_{2q}(c,c,p)$ vanishes at $\lambda$.

The vector in \eqref{eq:eigenvector} is nonzero because its first
coordinate is $1$. Its coordinates satisfy the eigenvalue equations
for the probability-one transitions. The two remaining equations are
\[
 \lambda^q-c=(1-c)w,\qquad
 (\lambda^q-c)w=(1-c)P_p(\lambda).
\]
These follow from the definition of $w$ and
$w^2=\Gamma_\lambda(s)=P_p(\lambda)$, completing the verification.
\end{proof}

For fixed $\lambda$, varying the probability vector $p$ lets
$P_p(\lambda)$ range over $K_q(\lambda)$, while varying the common
probability $c$ moves the curve point $\Gamma_\lambda(s)$ through
$s=(1-c)^{-1}$. Carath\'eodory's theorem in the plane says that a point
in a convex hull can be expressed as a weighted average of at most
three of its generating points. Here we can reduce this number to two.
When $\lambda^q\ne1$, choose the largest $s\in[1,s^*(\lambda)]$ for
which $\Gamma_\lambda(s)\in K_q(\lambda)$. This last intersection lies
on the boundary of the convex hull, where two of the complex numbers
$1,\lambda,\ldots,\lambda^{q-1}$ suffice; see
Figure~\ref{fig:transfer}. The proof below establishes this claim,
including when the convex hull is a segment or a point. This reduction
preserves the chosen eigenvalue $\lambda$; the other eigenvalues may change.

\begin{theorem}[Realization using three vertices]\label{thm:triangles}
Every $\lambda\in\Sigma_q$ is an eigenvalue of a balanced matrix
$B_{2q}(c,c,p)$ in which at most two entries of $p$ are positive.
Thus no more than two of the transition probabilities $(1-c)p_j$
from $B_{q-1}$ to the $A_j$ need be positive.

Equivalently, there are indices $q<M<N\le2q$ and nonnegative weights
$w_q,w_M,w_N$ summing to one such that $\lambda$ is an eigenvalue of
\[
 w_qU_q+w_MU_M+w_NU_N.
\]
Weights may be zero, so this includes realizations using fewer than
three vertices. Consequently,
\begin{equation}
 \Sigma_q=\bigcup_{q<M<N\le2q}\ \bigcup_{B\in\conv\{U_q,U_M,U_N\}}\spec B.\label{eq:triangles}
\end{equation}
\end{theorem}

\begin{proof}
If $\lambda^q=1$, choose $U_q$, which belongs to every triangular face
under consideration. Now fix $\lambda\in\Sigma_q$ with $\lambda^q\ne1$.
By the preceding proposition, the set
\[
 \{s\in[1,s^*(\lambda)]:\Gamma_\lambda(s)\in K_q(\lambda)\}
\]
is nonempty. It is closed because $\Gamma_\lambda$ is continuous and
$K_q(\lambda)$ is closed. Since it lies in a compact interval, it has
a largest element, denoted by $s_0$.

Suppose first that $K_q(\lambda)$ has nonempty interior in the complex
plane. If $s_0<s^*(\lambda)$ and $\Gamma_\lambda(s_0)$ were an interior
point, continuity would keep the curve inside the hull for slightly
larger $s$, contradicting the choice of $s_0$. If $s_0=s^*(\lambda)$,
then $|\Gamma_\lambda(s_0)|=1$. Such a point cannot be interior to a set
contained in the unit disk. Thus the last intersection lies on the
boundary of $K_q(\lambda)$, on an edge joining two of its vertices.
If the hull is a segment or a single point, at most two of its vertices
also suffice. In every case, the last intersection is a weighted average
of at most two of the listed powers.

Choose distinct indices $j,k$ and $t\in[0,1]$ such that
\[
 \Gamma_\lambda(s_0)=t\lambda^j+(1-t)\lambda^k.
\]
If only one power is needed, give a second index weight zero.
Set $p_j=t$, $p_k=1-t$, and all other entries of $p$ to zero, and take
$c=1-1/s_0$. The preceding proposition then gives a realizing matrix,
and its vertex representation is
\[
 B_{2q}(c,c,p)
 =cU_q+(1-c)tU_{2q-j}+(1-c)(1-t)U_{2q-k}.
\]
Ordering the last two vertex indices gives $q<M<N\le2q$ and the desired
triangular face.
\end{proof}

\begin{figure}[htbp]
\centering
\begin{tikzpicture}[x=3.5cm,y=3.5cm,font=\small,>=Stealth]
\filldraw[fill=black!8,draw=black!50] (1.0000000,0.0000000) -- (0.1471909,0.9293268) -- (-0.8419831,0.2735769) -- cycle;
\draw[black!55,dashed] plot coordinates {(1.0000000,0.0000000) (0.9657584,-0.0182345) (0.9319382,-0.0358297) (0.8985394,-0.0527857) (0.8655619,-0.0691025) (0.8330057,-0.0847801) (0.8008709,-0.0998184) (0.7691575,-0.1142175) (0.7378654,-0.1279773) (0.7069947,-0.1410980) (0.6765453,-0.1535794) (0.6465172,-0.1654215) (0.6169105,-0.1766245) (0.5877252,-0.1871882) (0.5589612,-0.1971127) (0.5306186,-0.2063980) (0.5026973,-0.2150440) (0.4751973,-0.2230508) (0.4481187,-0.2304183) (0.4214615,-0.2371467) (0.3952256,-0.2432358) (0.3694111,-0.2486857) (0.3440179,-0.2534963) (0.3190461,-0.2576677) (0.2944956,-0.2611999) (0.2703665,-0.2640929) (0.2466587,-0.2663466) (0.2233723,-0.2679611) (0.2005072,-0.2689364) (0.1780635,-0.2692724) (0.1560411,-0.2689692) (0.1344401,-0.2680268) (0.1132604,-0.2664452) (0.0925021,-0.2642243) (0.0721651,-0.2613642) (0.0522495,-0.2578649) (0.0327552,-0.2537263) (0.0136823,-0.2489485) (-0.0049693,-0.2435315) (-0.0231995,-0.2374752) (-0.0410084,-0.2307797) (-0.0583959,-0.2234450) (-0.0753620,-0.2154711) (-0.0919069,-0.2068579) (-0.1080303,-0.1976055) (-0.1237324,-0.1877139) (-0.1390132,-0.1771830) (-0.1538726,-0.1660129) (-0.1683106,-0.1542036) (-0.1823274,-0.1417550) (-0.1959227,-0.1286673) (-0.2090967,-0.1149403) (-0.2218494,-0.1005740) (-0.2341807,-0.0855686) (-0.2460906,-0.0699239) (-0.2575792,-0.0536399) (-0.2686465,-0.0367168) (-0.2792924,-0.0191544) (-0.2895169,-0.0009528) (-0.2993201,0.0178881) (-0.3087019,0.0373682) (-0.3176624,0.0574875) (-0.3262016,0.0782460) (-0.3343194,0.0996438) (-0.3420158,0.1216808) (-0.3492909,0.1443570) (-0.3561446,0.1676724) (-0.3625770,0.1916271) (-0.3685881,0.2162210) (-0.3741778,0.2414542) (-0.3793461,0.2673265) (-0.3840931,0.2938382) (-0.3884187,0.3209890) (-0.3923230,0.3487790) (-0.3958059,0.3772083) (-0.3988675,0.4062769) (-0.4015077,0.4359846) (-0.4037266,0.4663316) (-0.4055241,0.4973178) (-0.4069003,0.5289432) (-0.4078551,0.5612079) (-0.4083886,0.5941118) (-0.4085007,0.6276549) (-0.4081915,0.6618373) (-0.4074609,0.6966589) (-0.4063090,0.7321197) (-0.4047357,0.7682197) (-0.4027411,0.8049590) (-0.4003251,0.8423375) (-0.3974878,0.8803552) (-0.3942291,0.9190122)};
\draw[very thick] plot coordinates {(-0.4078588,0.5613695) (-0.4080535,0.5707016) (-0.4082139,0.5800859) (-0.4083399,0.5895223) (-0.4084316,0.5990109) (-0.4084888,0.6085515) (-0.4085118,0.6181443) (-0.4085003,0.6277893) (-0.4084545,0.6374863) (-0.4083744,0.6472355) (-0.4082598,0.6570368) (-0.4081109,0.6668903) (-0.4079277,0.6767959) (-0.4077101,0.6867536) (-0.4074581,0.6967634) (-0.4071718,0.7068254) (-0.4068511,0.7169395) (-0.4064960,0.7271057) (-0.4061066,0.7373241) (-0.4056828,0.7475946) (-0.4052247,0.7579172) (-0.4047322,0.7682920) (-0.4042053,0.7787189) (-0.4036441,0.7891979) (-0.4030485,0.7997290) (-0.4024185,0.8103123) (-0.4017542,0.8209477) (-0.4010555,0.8316352) (-0.4003225,0.8423749) (-0.3995551,0.8531667) (-0.3987533,0.8640106) (-0.3979172,0.8749067) (-0.3970467,0.8858549) (-0.3961419,0.8968552) (-0.3952026,0.9079076) (-0.3942291,0.9190122)};
\draw[very thick] (0.1471909,0.9293268)--(-0.8419831,0.2735769);
\fill (1.0000000,0.0000000) circle[radius=.012] node[right] {$1$};
\fill (0.1471909,0.9293268) circle[radius=.012] node[above] {$\lambda$};
\fill (-0.8419831,0.2735769) circle[radius=.012] node[left] {$\lambda^2$};
\node at (0.25,0.4) {$K_3(\lambda)$};
\fill (-0.4078588,0.5613695) circle[radius=.012];
\node[above left] at (-0.4078588,0.5613695) {$\Gamma_\lambda(1)$};
\draw[fill=white] (-0.3942291,0.9190122) circle[radius=.012] node[above left] {$\Gamma_\lambda(s^*)$};
\node at (0.45,-.42) {$\Gamma_\lambda(s)$};
\end{tikzpicture}
\caption{The geometric criterion for $q=3$ at the boundary eigenvalue
$\lambda=r_3(9/40)e^{9\pi i/20}$. The shaded triangle is
$K_3(\lambda)=\conv\{1,\lambda,\lambda^2\}$. The bold part of the curve
$\Gamma_\lambda(s)$ shows the range $1\le s\le s^*$. Only its starting
point, $\Gamma_\lambda(1)$, belongs to the triangle; it lies on the edge
$[\lambda,\lambda^2]$. Thus $p_0=0$, and only $p_1$ and $p_2$ are needed.
Since the intersection occurs at $s=1$, we have $c=0$, and $\lambda$ is
realized by a weighted average of $U_5$ and $U_4$.}
\label{fig:transfer}
\end{figure}

\subsection{Interior points and approximation}
Let $q\ge3$ and let $\lambda$ be nonreal with $|\lambda|<1$.
If $\Gamma_\lambda(s_0)$ lies strictly inside $K_q(\lambda)$ for a fixed
$s_0\ge1$, this inclusion survives small changes in $\lambda$.
The geometric criterion therefore realizes a whole neighborhood of
$\lambda$, making it an interior point of $\Sigma_q$.

For a balanced realization, positivity of every $p_j$ guarantees such
strict inclusion. Perturbing the probability vector to make all its
entries positive gives interior eigenvalues approaching any attainable
nonreal point in the open unit disk. The closure in
\eqref{eq:approximation} includes boundary points as limits of interior points.

\begin{lemma}[Interior eigenvalues and approximation]\label{lem:persistence}
Let $q\ge3$ and let $\lambda_0$ be nonreal with $|\lambda_0|<1$.
If, for some fixed $s_0\ge1$,
\[
 \Gamma_{\lambda_0}(s_0)\in\interior K_q(\lambda_0),
\]
then $\lambda_0\in\interior\Sigma_q$: a sufficiently small open disk
centered at $\lambda_0$ is contained in $\Sigma_q$.

In particular, if $\lambda_0$ is an eigenvalue of a balanced matrix
$B_{2q}(c,c,p)$ with $p_j>0$ for every $j$, then $\lambda_0$ is an
interior point of $\Sigma_q$.

Every attainable nonreal eigenvalue with modulus less than one is a
limit of interior points of $\Sigma_q$. Equivalently,
\begin{equation}
 \Sigma_q\cap(\D^\circ\setminus\R)\subseteq\overline{\interior\Sigma_q}.\label{eq:approximation}
\end{equation}
\end{lemma}

\begin{proof}
For $|u|=1$, define
\[
 h_\lambda(u)=\max_{0\le j<q}\Re(u\lambda^j).
\]
This is the largest value of the linear functional $z\mapsto\Re(uz)$
on $K_q(\lambda)$. A point $z$ belongs to the hull precisely when
$\Re(uz)\le h_\lambda(u)$ for every $|u|=1$.
The function $h_\lambda(u)$ is jointly continuous in $\lambda$ and $u$,
since it is the maximum of finitely many continuous functions.

Because $\Gamma_{\lambda_0}(s_0)$ lies strictly inside $K_q(\lambda_0)$,
the difference
$h_{\lambda_0}(u)-\Re(u\Gamma_{\lambda_0}(s_0))$ is positive for every
$|u|=1$. Its minimum on the compact unit circle is therefore a positive
number $\delta$, giving
\[
 \Re\bigl(u\Gamma_{\lambda_0}(s_0)\bigr)
 \le h_{\lambda_0}(u)-\delta\qquad(|u|=1).
\]
Joint continuity and compactness ensure that, for all $\lambda$
sufficiently close to $\lambda_0$, the same $s_0$ satisfies
\[
 \Re\bigl(u\Gamma_\lambda(s_0)\bigr)
 \le h_\lambda(u)-\delta/2\qquad(|u|=1).
\]
Choose this neighborhood inside the open unit disk. The inequalities put $\Gamma_\lambda(s_0)$ in
$K_q(\lambda)$, so the geometric criterion realizes every point in the
neighborhood. The required bound $s_0\le s^*(\lambda)$ follows from
this inclusion, as established above. Hence
$\lambda_0\in\interior\Sigma_q$.

For nonreal $\lambda$, the points $1,\lambda,\lambda^2$ do not lie on
one line, since
\[
 \Im\bigl(\overline{(\lambda-1)}(\lambda^2-1)\bigr)
 =|\lambda-1|^2\Im\lambda\ne0.
\]
Thus $K_q(\lambda)$ has nonempty planar interior when $q\ge3$.
A weighted average with every $p_j>0$ cannot lie on its boundary:
a line through a boundary point with the whole hull on one side would
force every positively weighted power onto that line, contradicting
the three noncollinear points. Therefore $P_p(\lambda)$ lies in the
interior of $K_q(\lambda)$. For the balanced matrix in the lemma,
$c<1$ and $\Gamma_\lambda((1-c)^{-1})=P_p(\lambda)$, so the first part
of the proof applies. This includes $c=0$, where $s_0=1$.

Finally, let $\lambda_0$ be any attainable nonreal eigenvalue with
$|\lambda_0|<1$, and choose a balanced realization $B_{2q}(c,c,p)$.
Necessarily $c<1$, since $c=1$ gives $U_q$, whose eigenvalues lie on
the unit circle. For $0<\eps<1$, replace $p$ by
\[
 p(\eps)=(1-\eps)p+\eps(1/q,\ldots,1/q).
\]
Every entry of $p(\eps)$ is at least $\eps/q>0$, and $p(\eps)\to p$.
The corresponding characteristic polynomials converge coefficient by
coefficient to that of the original matrix. Continuity of the roots of
monic polynomials with respect to their coefficients supplies eigenvalues
$\lambda_\eps\to\lambda_0$. For sufficiently small $\eps$, these remain
nonreal and have modulus less than one. The positive-weight assertion
just proved places them in $\interior\Sigma_q$, establishing the
approximation claim. The argument also applies when $\lambda_0$ is a
multiple eigenvalue.
\end{proof}

\subsection{Real eigenvalues, an excluded sector, and unit-circle points}
\begin{proposition}\label{prop:elementary}
The region $\Sigma_q$ is compact and contained in the unit disk. It is
symmetric about the real axis: whenever $\lambda$ is attainable, so is
$\overline\lambda$. Its real points are exactly $[-1,1]$. As stated in
\eqref{eq:unit}, its points on the unit circle are precisely the roots of
unity whose orders are at most $2q$.

No nonzero eigenvalue satisfies $0<|\Arg\lambda|<\pi/q$.
The three line segments $[0,1]$, $[0,e^{i\pi/q}]$, and
$[0,e^{-i\pi/q}]$ lie on the boundary of $\Sigma_q$.
\end{proposition}

\begin{proof}
The parameters $a,b,p$ range over a compact set, and the transition
matrix depends continuously on them. Stochasticity places every
eigenvalue in the unit disk. The equation
$\det(\lambda I-B_{2q}(a,b,p))=0$ therefore defines a closed subset of
the compact product of the parameter set and the unit disk. Its
projection onto the $\lambda$ coordinate is $\Sigma_q$, proving
compactness. Since the matrices have real entries, their nonreal
eigenvalues occur in conjugate pairs, giving symmetry about the real axis.

The matrix $cU_q+(1-c)U_{2q}=B_{2q}(c,c,e_0)$ has characteristic polynomial
\begin{equation}
 \chi_c(X)=(X^q-c)^2-(1-c)^2=(X^q-1)\bigl(X^q-(2c-1)\bigr).\label{eq:first-edge}
\end{equation}
For any $\lambda\in\D$ with $\lambda^q\in[-1,1]$, choose
$c=(1+\lambda^q)/2\in[0,1]$. Then $\chi_c(\lambda)=0$.
This realizes every real number in $[-1,1]$, and also every point on
the three segments in the proposition: on the positive real segment
$\lambda^q\in[0,1]$, and on the segments at angles $\pm\pi/q$,
$\lambda^q\in[-1,0]$. The unit-disk bound excludes any other real points.

To prove the angular gap, suppose
$\lambda=\rho e^{i\theta}$ with $\rho>0$ and $0<\theta<\pi/q$.
The powers $1,\lambda,\ldots,\lambda^{q-1}$ have angles
$0,\theta,\ldots,(q-1)\theta$. Since $(q-1)\theta<\pi$, every weighted
average of these powers is nonzero and has angle between $0$ and
$(q-1)\theta$. If $\lambda$ were attainable, the balancing proposition
would provide a matrix $B_{2q}(c,c,p)$ with eigenvalue $\lambda$.
This is a \emph{balanced realization} of $\lambda$: a permitted matrix
with equal probabilities $a=b=c$ that has the specified eigenvalue.
For $c<1$, its eigenvalue equation would require
\[
 P_p(\lambda)=\frac{(\lambda^q-c)^2}{(1-c)^2}.
\]
The number $\lambda^q$ has angle $q\theta\in(0,\pi)$. Subtracting
$c\ge0$ moves it horizontally to the left, so
$\phi=\Arg(\lambda^q-c)$ satisfies $q\theta\le\phi<\pi$.
If $\phi\le\pi/2$, its square has angle
$2\phi\ge2q\theta>(q-1)\theta$. If $\phi>\pi/2$, the principal angle
of its square is $2\phi-2\pi<0$. Neither case gives a possible angle
for $P_p(\lambda)$. The remaining case $c=1$ gives $U_q$, whose
eigenvalues satisfy $\lambda^q=1$ and cannot have $0<\theta<\pi/q$.
Symmetry gives the same exclusion below the real axis. Every point of
the three attainable segments can be approached by points in this
excluded sector, so the segments lie on the boundary.

Finally, roots satisfying $\lambda^q=1$ already have order at most $q$.
If $|\lambda|=1$ but $\lambda^q\ne1$, then $s^*(\lambda)=1$.
The geometric criterion gives
\[
 \lambda^{2q}=\sum_jp_j\lambda^j.
\]
A weighted average of points on the unit circle can lie on that circle
only when all positively weighted points coincide. Hence
$\lambda^{2q}=\lambda^j$ for at least one $j$ with $p_j>0$, so
$\lambda^{2q-j}=1$ and the order of $\lambda$ is at most $2q$.
Conversely, let $\lambda$ be a root of unity of order $k\le2q$.
The length $L=\ell_q(k)$ is a multiple of $k$ with $q\le L\le2q$,
so $\lambda^L=1$. If $L=q$, the matrix $U_q$ realizes it. If $L>q$,
the matrix $U_L$ has a cycle of length $L$, which realizes every root
satisfying $\lambda^L=1$. This proves the unit-circle classification.
\end{proof}

\section{Last intersections for boundary eigenvalues}\label{sec:contacts}
Throughout this section, assume $q\ge3$. We study boundary eigenvalues
with $\pi/q<\Arg\lambda<\pi$: they lie in the upper half-plane, between
the ray at angle $\pi/q$ and the negative real axis.
Appendix~\ref{app:four} treats the four-state case $q=2$ separately.

\subsection{Two possibilities for the last intersection}
For an attainable eigenvalue $\lambda$ with $\lambda^q\ne1$, define
\[
 s_{\max}=\max\{s\in[1,s^*(\lambda)]:\Gamma_\lambda(s)\in K_q(\lambda)\}.
\]
The maximum exists because the set is nonempty and closed and lies in a
compact interval. It selects the last point of the curve that belongs
to the convex hull as $s$ increases. Since $c=1-1/s$ increases with $s$,
$s_{\max}$ corresponds to the largest shared probability $c$ that a
balanced realization of this fixed $\lambda$ can have. The next
proposition identifies where this last intersection lies when $\lambda$
is a boundary eigenvalue with $|\lambda|<1$.

\begin{proposition}\label{prop:alternatives}
Let $\lambda$ be a boundary point of $\Sigma_q$ with $|\lambda|<1$ and
$\pi/q<\Arg\lambda<\pi$. Its last intersection has exactly one of the
following forms:
\begin{enumerate}
\item $s_{\max}=1$, and $\Gamma_\lambda(1)$ lies strictly between the
endpoints of an edge $[\lambda^j,\lambda^k]$ of $K_q(\lambda)$.
Taking the two averaging weights as $p_j,p_k$ and all other entries
zero gives $c=0$ and a realizing matrix on the base edge
$[U_{2q-j},U_{2q-k}]$.
\item $1<s_{\max}<s^*(\lambda)$, and
$\Gamma_\lambda(s_{\max})=\lambda^j$ is a vertex of $K_q(\lambda)$.
Taking $p_j=1$, all other entries zero, and
$c=1-1/s_{\max}\in(0,1)$ gives the realizing matrix
$cU_q+(1-c)U_{2q-j}$, on a lateral edge.
\end{enumerate}
\end{proposition}

\begin{proof}
The hull has nonempty planar interior because $q\ge3$ and $\lambda$ is
nonreal. As shown in the proof of Theorem~\ref{thm:triangles}, the last
intersection lies on its boundary. If $s_{\max}=1$, then
$\Gamma_\lambda(1)=\lambda^{2q}$. This cannot equal any listed power
$\lambda^j$, since $0<|\lambda|<1$ gives
$|\lambda|^{2q}<|\lambda|^j$ for $j<q$. In particular, it is not a
vertex, so it lies strictly between the endpoints of an edge. This is
the first case.

Suppose next that $1<s_{\max}<s^*(\lambda)$, but the last intersection
lies strictly between the endpoints of an edge. Choose real constants
$A,B,C_0$ so that
$H(z)=A\Re z+B\Im z+C_0$ is zero on the line through this edge and
negative in the interior of the hull. Put $g(s)=H(\Gamma_\lambda(s))$.
All other edge inequalities are strict at the intersection and remain
strict at nearby curve points. If $g'(s_{\max})<0$, slightly increasing
$s$ puts the curve inside the hull, contradicting maximality.
If $g'(s_{\max})>0$, slightly decreasing $s$ gives an interior
intersection. Lemma~\ref{lem:persistence} would then make $\lambda$
an interior point of $\Sigma_q$, contradicting the hypothesis.
Thus $g'(s_{\max})=0$.

The function $g$ is a real polynomial of degree at most two. Since
$g(s_{\max})=g'(s_{\max})=0$, it has the form
$g(s)=C(s-s_{\max})^2$. The polynomial formula gives
$\Gamma_\lambda(0)=1\in K_q(\lambda)$, so $g(0)\le0$ and hence $C\le0$.
If $C<0$, nearby curve points lie in the interior of the hull, giving
the same contradiction. If $C=0$, the entire curve lies on the line
$H=0$. Writing $d=\lambda^q-1$, its formula is
$\Gamma_\lambda(s)=1+2sd+s^2d^2$. The directions $d$ and $d^2$ must
therefore be parallel. Since $d\ne0$, this forces $d\in\R$, and the
line is the real axis.

But $\lambda^q\in\R$ implies
$\Arg\lambda=m\pi/q$ for an integer $2\le m\le q-1$.
Set $j_0=\lfloor q/m\rfloor+1\le q-1$. Then $\lambda$ lies above the
real axis, while
$\pi<j_0m\pi/q<2\pi$ puts $\lambda^{j_0}$ below it. The hull therefore
has points on both sides of the real axis, so that axis cannot be the
line through an edge. This excludes an edge-interior intersection for
$1<s_{\max}<s^*(\lambda)$; the intersection must be a vertex.

It remains to rule out $s_{\max}=s^*(\lambda)$. Since $|\lambda|<1$,
the only point of $K_q(\lambda)$ with modulus one is $1$.
At $s^*(\lambda)$ the curve has modulus one, so
$\Gamma_\lambda(s^*(\lambda))=1$. Set
$w=1+s^*(\lambda)(\lambda^q-1)$. Then $w^2=1$. The choice $w=1$
would give $\lambda^q=1$, which is impossible here. Hence $w=-1$,
and $u=\lambda^q=1-2/s^*(\lambda)<1$ is real.

Choose $m$ and $j_0=\lfloor q/m\rfloor+1$ as above. Then $\lambda$
and $\lambda^{j_0}$ lie on opposite sides of the real axis. Both have real part less than $1$, so the triangle with
vertices $1,\lambda,\lambda^{j_0}$ contains real points immediately
to the left of $1$ in its interior. For small $\eta>0$,
\[
 \Gamma_\lambda(s^*(\lambda)-\eta)
 =1-2\eta(1-u)+\eta^2(1-u)^2.
\]
These are precisely such real points. Taking $\eta$ small enough also
keeps $s^*(\lambda)-\eta\ge1$, because $|\lambda|<1$ implies
$s^*(\lambda)>1$. We obtain an interior intersection, again
contradicting Lemma~\ref{lem:persistence}. The two stated cases are
therefore exhaustive. Their matrix realizations follow by choosing the
two edge weights in the first case and the single vertex weight in the second.
\end{proof}

\subsection{Line segments, convexity, and Farey angles}
The boundary calculations for both base and lateral edges use a
straight line segment joining two complex numbers. Requiring $1$ to lie
strictly between its endpoints gives an equation for the radius $\rho$.
The next lemma also determines the weights in this average.

\begin{lemma}\label{lem:chord}
Let $u,v>0$ be real exponents, and let $\alpha,\beta>0$ be angles with
$\alpha+\beta<\pi$. The equation \eqref{eq:scalar}, namely
\[
 \rho^v\sin\alpha+\rho^u\sin\beta=\sin(\alpha+\beta),
\]
has exactly one solution $\rho\in(0,1)$.
For $R=\rho^{-1}$, this equation is equivalent to requiring the complex
number $1$ to lie strictly between the endpoints of the straight line
segment joining $R^ue^{i\alpha}$ and $R^ve^{-i\beta}$:
\begin{equation}
 1\in\bigl(R^ue^{i\alpha},R^ve^{-i\beta}\bigr).\label{eq:chord}
\end{equation}
More explicitly,
\[
 1=(1-t)R^ue^{i\alpha}+tR^ve^{-i\beta},\qquad
 t=\frac{\rho^v\sin\alpha}{\sin(\alpha+\beta)}\in(0,1).
\]
Thus $t$ is the weight of the second endpoint and $1-t$ that of the first.
If $u,v,\alpha,\beta$ vary real-analytically with real parameters, then
$\rho$ and $t$ do as well, provided the stated inequalities remain valid.
\end{lemma}

\begin{proof}
The two endpoints have imaginary parts of opposite signs. Solving for the
average with zero imaginary part shows that their segment meets the
positive real axis at
\[
 \frac{R^{u+v}\sin(\alpha+\beta)}{R^u\sin\alpha+R^v\sin\beta}.
\]
Setting this expression equal to $1$ and substituting $R=\rho^{-1}$ gives
\eqref{eq:scalar}.

The left-hand side of that equation is continuous and strictly increasing
for $\rho>0$, with limit $0$ as $\rho\to0$. Its right-hand side is positive.
At $\rho=1$, the left-hand side exceeds the right-hand side by
\[
 \sin\alpha+\sin\beta-\sin(\alpha+\beta)
 =4\sin(\alpha/2)\sin(\beta/2)\sin((\alpha+\beta)/2)>0.
\]
The intermediate value theorem therefore gives exactly one solution in
$(0,1)$.

At this solution, cancellation of the imaginary parts gives the averaging
weights
\[
 t=\frac{\rho^v\sin\alpha}{\sin(\alpha+\beta)},\qquad
 1-t=\frac{\rho^u\sin\beta}{\sin(\alpha+\beta)}.
\]
Both are positive, and they sum to one by \eqref{eq:scalar}, so the point
$1$ lies strictly between the two endpoints.
Finally, the derivative of the equation's left-hand side with respect to
$\rho$ is
\[
 v\rho^{v-1}\sin\alpha+u\rho^{u-1}\sin\beta>0.
\]
The implicit-function theorem gives real-analytic dependence of $\rho$ on
the data; the formula for $t$ gives the same dependence for the weight.
\end{proof}

The powers to be compared below are $z^q,z^{q+1},\ldots,z^{2q}$ of a fixed
complex number $z=Re^{i\theta}$. These are ordinary integer powers:
$z^L=R^Le^{iL\theta}$. Thus the $L$th power has modulus $R^L$ and angle
$L\theta$, taken modulo $2\pi$. To determine on which side of a line each
power lies, we describe where that line meets each ray from the origin. Consider a line through $1$
that does not pass through $0$. Choose a unit vector perpendicular to the
line and pointing towards it from the origin, and let
$\eta\in(-\pi/2,\pi/2)$ be its angle. The line then has equation
\[
 \Re(e^{-i\eta}w)=\cos\eta.
\]
A ray at angle $\varphi$ meets the line at a positive radius precisely when
$\cos(\varphi-\eta)>0$. Substituting $w=re^{i\varphi}$ into the line equation
gives that radius as $r=\cos\eta/\cos(\varphi-\eta)$.
Writing $\psi(\varphi)=\log r$, we obtain, on the interval of such angles
containing $0$,
\begin{equation}
 \psi(\varphi)=\log\cos\eta-\log\cos(\varphi-\eta),\qquad
 \psi(0)=0,\qquad \psi''(\varphi)=\sec^2(\varphi-\eta)>0.\label{eq:logradial}
\end{equation}
Thus the logarithm of the intersection radius is strictly convex as a
function of the angle. We will use this fact to determine on which side
of the line the other powers lie.

\begin{lemma}\label{lem:convex}
Consider the line through the endpoints in \eqref{eq:chord}, with $R>1$.
Let $\psi(\varphi)$ be the logarithm of the radius at which the ray at angle
$\varphi$ meets this line. For any $a,b\in[0,1]$,
\begin{equation}
 \psi(a\alpha-b\beta)\le a\psi(\alpha)+b\psi(-\beta),\label{eq:convex}
\end{equation}
and equality holds exactly when
$(a,b)=(0,0),(1,0)$ or $(0,1)$.
The same comparison holds after reflecting the configuration across the
real axis, with all angles replaced by their negatives.
\end{lemma}

\begin{proof}
The line meets the rays at angles $\alpha$ and $-\beta$ at its two
endpoints, whose radii are $R^u$ and $R^v$. Hence
\[
 \psi(\alpha)=u\log R>0,\qquad \psi(-\beta)=v\log R>0.
\]
Set $y=a\alpha-b\beta$. Since $a,b\in[0,1]$, we have
$-\beta\le y\le\alpha$.
If $y\ge0$, convexity on $[0,\alpha]$ and $\psi(0)=0$ give
\[
 \psi(y)\le\frac y\alpha\psi(\alpha)
 \le a\psi(\alpha)
 \le a\psi(\alpha)+b\psi(-\beta).
\]
For $b>0$, the last inequality is strict. If $b=0$, then $y=a\alpha$,
and strict convexity makes the first inequality strict unless $a=0$ or
$a=1$.

If $y<0$, convexity on $[-\beta,0]$ instead gives
\[
 \psi(y)\le\frac{-y}{\beta}\psi(-\beta)
 \le b\psi(-\beta)
 \le a\psi(\alpha)+b\psi(-\beta).
\]
For $a>0$, the last inequality is strict. If $a=0$, then $y=-b\beta$,
and strict convexity leaves equality only at $b=1$, since $y<0$.
Together, the two cases give precisely the three stated equality pairs.
Reflection across the real axis sends $\varphi$ to $-\varphi$ and preserves
strict convexity, so the same argument applies to the reflected
configuration.
\end{proof}

Let $h/k<r/s$ be consecutive fractions in $\F_q$, the ordered list of
reduced fractions in $[1/(2q),1/2]$ with denominators at most $2q$.
Let $M=uk$ and $N=vs$ be the smallest multiples of $k$ and $s$, respectively,
that are at least $q$. The Farey relations and these choices give
\begin{equation}
 kr-hs=1,\qquad k+s>2q,\qquad
 M=uk=k\lceil q/k\rceil,\qquad N=vs=s\lceil q/s\rceil.\label{eq:farey-data}
\end{equation}
Farey fractions also occur in Dubuc and Malik's study of convex hulls of
powers \cite{dubuc}.
Write $\{y\}=y-\lfloor y\rfloor$ for the fractional part of $y$.
For $z=Re^{2\pi ix}$, the angle of $z^L$, measured counterclockwise from the
positive real axis and reduced to $[0,2\pi)$, is $2\pi\{Lx\}$.
The next lemma compares these angles for all integer exponents
$q\le L\le2q$.

\begin{lemma}\label{lem:phases}
For $x\in(h/k,r/s)$ and every integer $q\le L\le2q$, the fractional parts
satisfy
\begin{align}
 \{Mx\}&\le\{Lx\}\le\{Nx\},\label{eq:phase-order}\\
 \frac{\{Mx\}}M&\le\frac{\{Lx\}}L,\qquad
 \frac{1-\{Nx\}}N\le\frac{1-\{Lx\}}L.\label{eq:phase-ratios}
\end{align}
In \eqref{eq:phase-order}, equality on the left occurs exactly when $L=M$,
and equality on the right exactly when $L=N$.
Thus, among the powers $z^q,\ldots,z^{2q}$, $z^M$ is closest to the
positive real axis when measured counterclockwise, and $z^N$ when measured
clockwise. The inequalities in \eqref{eq:phase-ratios} compare these
angular distances after division by the corresponding exponents.
The integer multipliers $u$ and $v$ also satisfy
\[
 2u\le s,\qquad 2v\le k.
\]
These two bounds cannot both be equalities. Moreover, $2u=s$ is possible
only when $s=2$, and $2v=k$ only when $k=2$.
\end{lemma}

\begin{proof}
Set $A=kx-h$. The Farey identity $kr-hs=1$ gives
$0<A<1/s$. Since $q<(k+s)/2\le ks$, we have
$u=\lceil q/k\rceil\le s$, and hence $0<uA<1$. Thus
$Mx=uh+uA$ and $\{Mx\}=uA$.

For an integer $q\le L\le2q$, let $n=\lfloor Lx\rfloor$.
No fraction with denominator at most $2q$ lies strictly between the two
Farey neighbors. Since $n/L\le x<r/s$, it follows that $n/L\le h/k$.
Consequently,
\[
 H=hL-kn\ge0,\qquad \{Lx\}=\frac{LA+H}{k}.
\]
If $L\ge M=uk$, this expression is at least $uA$, and is strictly greater
when $L>M$.
If $L<M$, put $K=rL-sn>0$. The Farey identity gives
\[
 kK-sH=L.
\]
Since $kK\ge L\ge q>(u-1)k$, the integer $K$ satisfies $K\ge u$.
It follows that $sH=kK-L\ge M-L$. Combining this with $A<1/s$ yields
\[
 \{Lx\}-\{Mx\}
 =\frac{H-(M-L)A}{k}>0.
\]
This proves $\{Mx\}\le\{Lx\}$, with equality exactly when $L=M$.
The corresponding ratio bound follows directly from $n/L\le h/k$:
\[
 \frac{\{Lx\}}L=x-\frac nL
 \ge x-\frac hk=\frac{\{Mx\}}M.
\]

To obtain the upper bounds, replace $x$ by $1-x$ and the Farey pair by
$(s-r)/s<(k-h)/k$. This interchanges $M$ and $N$.
Moreover, $Lx$ cannot be an integer, since that would put the fraction
$x$ itself between the Farey neighbors with denominator at most $2q$.
Thus $\{L(1-x)\}=1-\{Lx\}$, and the two lower bounds just proved become
\[
 \{Lx\}\le\{Nx\},\qquad
 \frac{1-\{Nx\}}N\le\frac{1-\{Lx\}}L.
\]
Equality in the upper bound $\{Lx\}\le\{Nx\}$ occurs exactly when $L=N$.
Indeed, after $x$ is replaced by $1-x$, the left Farey denominator is $s$,
whose smallest multiple at least $q$ is $N$; the lower-bound argument
therefore singles out $L=N$.

For the multiplier bounds, note that $k,s\ge2$. The elementary inequalities
\[
 (k+s)/2\le k\lfloor s/2\rfloor,\qquad
 (k+s)/2\le s\lfloor k/2\rfloor
\]
hold except for the unordered pair $\{k,s\}=\{2,3\}$.
That pair is excluded by $k+s>2q\ge6$.
Together with $q<(k+s)/2$, these inequalities give
$u\le\lfloor s/2\rfloor$ and $v\le\lfloor k/2\rfloor$.
The equalities $2u=s$ and $2v=k$ cannot both hold: they would make $s$ and
$k$ even, whereas $kr-hs=1$ forces them to be coprime.

Finally, suppose $2v=k$. Then $k$ is even and $s$ is odd.
Since $v=\lceil q/s\rceil$,
\[
 (v-1)s<q<(2v+s)/2,
 \qquad\text{so}\qquad 2v(s-1)<3s.
\]
If $v\ge2$, then $s\ge3$ and this last inequality forces $s=3$ and $v=2$,
hence $k=4$. But $v=\lceil q/3\rceil=2$ requires $q\ge4$, contradicting
$k+s=7>2q$.
Therefore $v=1$ and $k=2$. Interchanging $k,u$ with $s,v$ gives
$2u=s\Rightarrow s=2$.
\end{proof}

To study an eigenvalue inside the unit disk using powers of a number
outside it, take the reciprocal of its complex conjugate.
For $\lambda=\rho e^{2\pi ix}$ with $0<\rho<1$, define
\begin{equation}
 z=\overline\lambda^{-1}=Re^{2\pi ix},\qquad
 Q_z=\conv\{z^q,z^{q+1},\ldots,z^{2q}\},\qquad R>1.\label{eq:reciprocal}
\end{equation}
Here $z=1/\overline\lambda$ and $R=1/\rho$.
The transformation keeps the angle $2\pi x$ and replaces the radius by its
reciprocal: smaller values of $\rho$ correspond to larger values of $R$.
The set $Q_z$ is the convex hull of the listed integer powers of $z$.

\begin{proposition}\label{prop:chord}
Let $h/k<r/s$ be consecutive fractions in $\F_q$, with $M,N$ as in
\eqref{eq:farey-data}, and fix $x\in(h/k,r/s)$.
There is exactly one $R>1$ for which $1$ lies strictly between $z^M$ and
$z^N$ on the straight line segment joining them:
\[
 1\in(z^M,z^N).
\]
At this radius, $[z^M,z^N]$ is an edge of $Q_z$.
Every other listed power $z^L$, with $L\ne M,N$, lies strictly on the same
side of the line through this edge as the origin.
Conversely, for any $R>1$, if a segment joining two of the listed powers
contains $1$ and its line leaves all of $Q_z$ on one side, then its
endpoints must be $z^M$ and $z^N$.
\end{proposition}

Figure~\ref{fig:reciprocal-chord} shows this segment for $q=3$.

\begin{figure}[htbp]
\centering
\begin{tikzpicture}[x=2.25cm,y=2.25cm,>=Stealth,font=\small]
\path[use as bounding box] (-1.50,-1.63) rectangle (1.62,1.98);
\fill[black!7] plot coordinates {(-0.54500587,-1.06963425) (1.03219924,-0.74993664) (0.95883166,0.95883166) (-0.84708603,1.16591390)} -- cycle;
\draw[->,black!40] (-1.35,0)--(1.42,0) node[right] {$\Re$};
\draw[->,black!40] (0,-1.45)--(0,1.73) node[above] {$\Im$};
\draw[black!60] plot coordinates {(-0.54500587,-1.06963425) (1.03219924,-0.74993664) (0.95883166,0.95883166) (-0.84708603,1.16591390) (-0.54500587,-1.06963425)};
\draw[blue!60!black,dashed] (1.05787789,-1.34800555)--(0.93315300,1.55690056);
\draw[blue!60!black,very thick] (1.03219924,-0.74993664)--(0.95883166,0.95883166);
\fill (-0.54500587,-1.06963425) circle[radius=.018] node[below left,inner sep=3pt] {$z^3$};
\fill (1.03219924,-0.74993664) circle[radius=.018] node[below right,inner sep=3pt] {$z^4$};
\fill (0.95883166,0.95883166) circle[radius=.018] node[above right,inner sep=3pt] {$z^5$};
\fill (-0.84708603,1.16591390) circle[radius=.018] node[above left,inner sep=3pt] {$z^6$};
\fill (0,0) circle[radius=.018] node[below left,inner sep=3pt] {$0$};
\fill[blue!60!black] (1,0) circle[radius=.022];
\node[right,fill=white,inner sep=1pt,xshift=2pt] at (1,0) {$1$};
\node at (-.45,.24) {$Q_z$};
\end{tikzpicture}
\caption{The segment in Proposition~\ref{prop:chord}, for $q=3$ and
$x=9/40$, between the Farey fractions $1/5$ and $1/4$.
Take $\lambda=r_3(9/40)e^{2\pi i(9/40)}$ and $z=1/\overline\lambda$.
The blue segment $[z^4,z^5]$ contains $1$ strictly between its endpoints.
Its line leaves the other listed powers $z^3,z^6$ and the origin on the
same side, so it is an edge of $Q_z=\conv\{z^3,z^4,z^5,z^6\}$.
The radius is approximately $r_3(9/40)=0.940911$.}
\label{fig:reciprocal-chord}
\end{figure}

\begin{proof}
First, the origin lies in the interior of $Q_z$.
Otherwise, a line through the origin would put all the listed powers in
one closed half-plane. Their angles would then have representatives in
an interval of length $\pi$.
Write $\theta=2\pi x\in(\pi/q,\pi)$.
Consecutive powers have angles differing by $\theta$ modulo $2\pi$.
The only such difference in $[-\pi,\pi]$ is $\theta$, since
$\theta-2\pi<-\pi$. The chosen angles from $z^q$ to $z^{2q}$ would therefore
span $q\theta>\pi$, a contradiction.

Set $A=kx-h$ and $B=r-sx$. Then $A,B>0$ and $sA+kB=1$.
Define
\[
 \alpha=2\pi uA,\qquad \beta=2\pi vB.
\]
The bounds $2u\le s$ and $2v\le k$ from Lemma~\ref{lem:phases}, which cannot
both be equalities, give
\[
 0<\alpha+\beta=2\pi(uA+vB)<\pi(sA+kB)=\pi.
\]
Since $M=uk$ and $N=vs$,
\[
 z^M=R^Me^{i\alpha},\qquad z^N=R^Ne^{-i\beta}.
\]
Lemma~\ref{lem:chord}, with exponents $M,N$, now gives exactly one $R>1$
for which their segment contains $1$ strictly between its endpoints.

At this radius, let $\psi(\varphi)$ be the logarithm of the radius where
the ray at angle $\varphi$ meets the line through $z^M$ and $z^N$.
Consider another power $z^L$ in the upper half-plane, and let
$\varphi\in(0,\pi)$ be its angle.
Lemma~\ref{lem:phases} gives $\varphi>\alpha$ and
$\alpha/M\le\varphi/L$.
If this ray meets the line at positive radius, strict convexity, with
$\alpha$ between $0$ and $\varphi$, gives
\[
 \psi(\varphi)>\frac\varphi\alpha\psi(\alpha)
 =\frac\varphi\alpha M\log R\ge L\log R.
\]
Thus the line is met at a radius greater than $R^L$: the point $z^L$
lies between the origin and the line.
For a power in the lower half-plane, write its angle as $-\gamma$.
The other bounds in Lemma~\ref{lem:phases} give $\gamma>\beta$ and
$\beta/N\le\gamma/L$. If its ray meets the line at positive radius,
strict convexity on the negative-angle interval gives
\[
 \psi(-\gamma)>\frac\gamma\beta\psi(-\beta)
 =\frac\gamma\beta N\log R\ge L\log R.
\]
This point also lies between the origin and the line.
Any ray with no positive intersection stays entirely on the origin's
side. This includes the negative real ray, since the line's only
real-axis intersection is $1$.
No listed power lies on the positive real ray: since $x$ is strictly
between consecutive Farey fractions and $L\le2q$, $Lx$ is not an integer.
Hence every other listed power lies strictly on the origin's side, and
$[z^M,z^N]$ is an edge of $Q_z$.

Conversely, suppose a segment joining two listed powers contains $1$ and
its line leaves $Q_z$ on one side.
Since the origin is inside $Q_z$, it lies strictly on that side.
The point $1$ cannot be an endpoint, because every listed power has
modulus greater than $1$. The endpoints must therefore lie in opposite
half-planes above and below the real axis.
Let the upper endpoint be $z^m$, with angle $\varphi\in(0,\pi)$, and let
$\psi$ now denote the logarithmic radius function for this line.
If $m\ne M$, Lemma~\ref{lem:phases} gives $\varphi>\alpha$ and
$\alpha/M\le\varphi/m$. Strict convexity on $[0,\varphi]$ gives
\[
 \psi(\alpha)<\frac\alpha\varphi\psi(\varphi)
 =\frac\alpha\varphi m\log R\le M\log R.
\]
This puts $z^M$ beyond the line, on the side opposite the origin,
contradicting the assumption that the line leaves all of $Q_z$ on one
side. Thus $m=M$.
If the lower endpoint is $z^n$, with angle $-\gamma$, and $n\ne N$, then
$\gamma>\beta$ and $\beta/N\le\gamma/n$ give, in the same way,
\[
 \psi(-\beta)<\frac\beta\gamma\psi(-\gamma)
 =\frac\beta\gamma n\log R\le N\log R.
\]
This would put $z^N$ on the side opposite the origin. Therefore the lower
endpoint is $z^N$.
\end{proof}

\subsection{Boundary eigenvalues realized by base edges}
Taking the complex conjugate of $\Gamma_\lambda(s)\in K_q(\lambda)$ and
multiplying by $z^{2q}$ gives the equivalent condition $W_z(s)\in P_z$,
where
\begin{equation}
 W_z(s)=\bigl(s-(s-1)z^q\bigr)^2,\qquad
 P_z=\conv\{z^{q+1},\ldots,z^{2q}\}.\label{eq:reciprocal-transfer}
\end{equation}
Here $s$ again denotes the real parameter of $\Gamma_\lambda$, with
$1\le s\le s^*(\lambda)$.
Indeed,
\[
 W_z(s)=z^{2q}\overline{\Gamma_\lambda(s)},\qquad
 P_z=z^{2q}\overline{K_q(\lambda)}.
\]
Adding $z^q$ to the points defining $P_z$ gives $Q_z$.
In particular, $W_z(1)=1$, so the case $s_{\max}=1$ concerns the location
of $1$ on the boundary of $P_z$.

\begin{proposition}\label{prop:base}
Let $\lambda$ satisfy Proposition~\ref{prop:alternatives}, and write
$\Arg\lambda=2\pi x$. Suppose $x$ lies strictly between consecutive
fractions $f=h/k$ and $g=r/s$ in $\F_q$, and let $M,N$ be the associated
integers in \eqref{eq:farey-data}.
If $s_{\max}=1$, then $M,N>q$.
Moreover, the only nonzero coefficients in
\[
 \Gamma_\lambda(1)=\sum_{j=0}^{q-1}p_j\lambda^j
\]
are $p_{2q-M}$ and $p_{2q-N}$, both positive.
Since $c=1-1/s_{\max}=0$, the corresponding realizing matrix is
\[
 B(0,0,p)=p_{2q-M}U_M+p_{2q-N}U_N.
\]
Thus $\lambda$ is realized by a weighted average of the two vertices on
the base edge $[U_M,U_N]$.
\end{proposition}

\begin{proof}
By Proposition~\ref{prop:alternatives}, the point $\Gamma_\lambda(1)$ lies
strictly between two powers on an edge of $K_q(\lambda)$.
After conjugation and multiplication by $z^{2q}$, this becomes
\[
 1=(1-t)z^m+tz^n,\qquad 0<t<1,\qquad q<m,n\le2q,
\]
and the line through these endpoints leaves $P_z$ on one side.
We must show that $z^q$ is on that same side.

Choose a real-valued function
$H(w)=A\Re w+B\Im w+C$ that vanishes on the line and is negative inside
$P_z$. Since $H(1)=0$, $W_z(1)=1$ and $W_z'(1)=2(1-z^q)$,
\[
 \left.\frac{d}{ds}H(W_z(s))\right|_{s=1}=-2H(z^q).
\]
If $H(z^q)>0$, this derivative is negative, so $W_z(s)$ moves to the
polygon's side of the line for $s$ just greater than $1$.
The point $1$ lies strictly between the endpoints of the edge, so all
other edge inequalities are already strict there. They remain strict
for nearby $s$, and hence $W_z(s)$ enters the interior of $P_z$.
Undoing the change of coordinates puts $\Gamma_\lambda(s)$ in the interior
of $K_q(\lambda)$. Lemma~\ref{lem:persistence} would then make $\lambda$
an interior eigenvalue, contradicting its position on the boundary.

Therefore $H(z^q)\le0$. Adding $z^q$ to the points defining $P_z$ keeps
all of $Q_z$ on the same side of the line.
Proposition~\ref{prop:chord} now gives $\{m,n\}=\{M,N\}$.
Since $m,n>q$, both $M$ and $N$ exceed $q$.
That proposition also puts every other listed power strictly off the
line. Thus only the two endpoint powers can receive positive weights
in an average equal to $1$, and both weights are positive because $1$
lies strictly between them. Transforming back gives exactly the two
positive entries $p_{2q-M}$ and $p_{2q-N}$.
\end{proof}

\subsection{Boundary eigenvalues realized by lateral edges}
We now consider the second case in Proposition~\ref{prop:alternatives}:
the last intersection is at a vertex $\lambda^j$ of $K_q(\lambda)$, with
$0\le j\le q-1$. Write $s_0=s_{\max}$ and set
\begin{equation}
 \Gamma_\lambda(s_0)=\lambda^j=\xi^2,\qquad
 \xi=1+s_0(\lambda^q-1),\qquad1<s_0<s^*(\lambda).\label{eq:one-site}
\end{equation}
The equation $\xi^2=\lambda^j$ alone does not distinguish the two square
roots $\xi$ and $-\xi$. Their signs matter for the geometry of the curve.
Indeed, differentiating $\Gamma_\lambda$ gives
\[
 \Gamma_\lambda'(s_0)=\frac{2}{s_0}(\xi^2-\xi).
\]
Thus the tangent line at $\xi^2$ is the straight line through $\xi$ and
$\xi^2$:
\[
 \Lambda=\{\xi^2+t(\xi^2-\xi):t\in\mathbb R\}.
\]
Using the opposite square root gives the line through $-\xi$ and $\xi^2$.
The next proposition shows that, for the root associated with the last
intersection, the whole polygon $K_q(\lambda)$ and the origin must lie
on the same side of $\Lambda$ (Figure~\ref{fig:root}).
This geometric condition selects the solutions of the edge's
characteristic equation that can describe the boundary in a Farey
interval. Rejected solutions are excluded only as candidates for the
last intersection describing that boundary; they may still be
eigenvalues of matrices in the family.

\begin{proposition}\label{prop:bridge}
Suppose the last intersection in the second case of
Proposition~\ref{prop:alternatives} is
$\Gamma_\lambda(s_0)=\lambda^j$, with $s_0=s_{\max}$ and $\xi$
as defined in \eqref{eq:one-site}. Then the tangent line $\Lambda$
has all of $K_q(\lambda)$ in one closed half-plane.
The origin lies strictly on the same side of $\Lambda$ as the interior of
the polygon.
Moreover, $\xi$ is nonreal and differs from every listed power
$1,\lambda,\ldots,\lambda^{q-1}$.
\end{proposition}

\begin{proof}
For small real $h$, expanding the square in $\Gamma_\lambda$ gives
\begin{equation}
 \Gamma_\lambda(s_0+h)=\xi^2+\frac{2h}{s_0}(\xi^2-\xi)
 +\frac{h^2}{s_0^2}(\xi-1)^2.\label{eq:tangent-expansion}
\end{equation}
Describe the polygon using one function for each edge:
\[
 E_\ell(w)=a_\ell\Re w+b_\ell\Im w+c_\ell,\qquad
 K_q(\lambda)=\{w:E_\ell(w)\ge0\text{ for every edge }\ell\}.
\]
Choose each function to vanish on the line through its edge and to be
positive inside the polygon. Thus a point is in the polygon's interior
exactly when every $E_\ell(w)>0$.
The tangent direction is nonzero: $\xi=0$ would give $\lambda^j=0$, and
$\xi=1$ would give $\lambda^q=1$, both impossible here.
If the polygon had points on both sides of $\Lambda$, one of the two
directions along this line would point strictly into the polygon from
$\xi^2$. For one sign of $h$, all the functions $E_\ell$ would therefore
be positive at the constant-plus-linear approximation in
\eqref{eq:tangent-expansion}. Those that vanish at $\xi^2$ become positive
in proportion to $|h|$; the others remain bounded away from zero.
The quadratic term changes each value by an amount bounded by a constant
times $h^2$. Hence all values remain positive for sufficiently small $h$,
and the full curve enters the polygon's interior.
Both signs are admissible because $1<s_0<s^*(\lambda)$.
Lemma~\ref{lem:persistence} would make $\lambda$ an interior eigenvalue,
contradicting its position on the boundary.
Thus $\Lambda$ leaves the polygon on one side.

Suppose next that $\xi$ is real. Then $\Lambda$ is the real axis and
$\lambda^j=\xi^2\ge0$.
Write $\theta=\Arg\lambda\in(0,\pi)$.
Since $\lambda$ lies above the real axis, all listed powers would have to
lie on or above it.
If $(q-1)\theta>\pi$, the first listed power whose angle exceeds $\pi$
has angle less than $2\pi$, and hence lies below the real axis.
Therefore $(q-1)\theta\le\pi$.
Every nonconstant listed power then has angle in $(0,\pi]$, so none is
nonnegative real. It follows that $j=0$ and $\xi^2=1$.
The choice $\xi=1$ is already excluded. If $\xi=-1$, its definition gives
$\lambda^q=1-2/s_0$, and the formula for $s^*(\lambda)$ gives
$s^*(\lambda)=s_0$, contradicting the strict upper bound on $s_0$.
Hence $\Im\xi\ne0$.

To distinguish the two sides of $\Lambda$, put
\[
 e=\xi-\xi^2,\qquad N(w)=\Im\bigl((w-\xi^2)\overline e\bigr).
\]
The function $N$ is zero on $\Lambda$, and its sign identifies which side
of the line contains a point. Calculation gives
\begin{equation}
 \begin{aligned}
 N(1)&=-|1-\xi|^2\Im\xi,&N(0)&=-|\xi|^2\Im\xi,\\
 N\bigl(\xi^2+(\xi-1)^2\bigr)&=-|1-\xi|^2\Im\xi.
 \end{aligned}\label{eq:normal-signs}
\end{equation}
Since $1$ belongs to the polygon and $N(1)\ne0$, its sign identifies the
polygon's side of $\Lambda$.
The third identity shows that the quadratic displacement $(\xi-1)^2$
points to that same side.

Now suppose $\xi=\lambda^k$ for a listed exponent $k$.
Taking moduli in $\lambda^j=\xi^2$ gives
$|\lambda|^j=|\lambda|^{2k}$, so $j=2k$ because $0<|\lambda|<1$.
For $h<0$ sufficiently small, the constant and linear terms in
\eqref{eq:tangent-expansion} give the point
\[
 w_h=\xi^2+\frac{-2h}{s_0}(\xi-\xi^2),\qquad 0<\frac{-2h}{s_0}<1.
\]
This point lies strictly between the two listed powers $\xi^2$ and $\xi$
on one edge. Denote that edge's function by $E_{\ell_0}$, so
$E_{\ell_0}(w_h)=0$.
The sign calculation \eqref{eq:normal-signs} shows that adding the
quadratic term makes $E_{\ell_0}$ positive.
For every other edge, the functions defined above satisfy
\[
 E_\ell(w_h)\ge C_\ell|h|\qquad(\ell\ne\ell_0)
\]
for some constants $C_\ell>0$ and all sufficiently small negative $h$.
Indeed, $w_h$ moves along the edge into its interior, and each
$E_\ell(w_h)$ is affine in $h$ and strictly positive there.
Adding the quadratic term changes $E_\ell(w_h)$ by at most $D_\ell h^2$
for some constant $D_\ell\ge0$. Consequently,
\[
 E_\ell\bigl(\Gamma_\lambda(s_0+h)\bigr)
 \ge C_\ell|h|-D_\ell h^2>0\qquad(\ell\ne\ell_0)
\]
when $h$ is sufficiently small.
All edge functions are therefore positive at $\Gamma_\lambda(s_0+h)$,
placing it in the polygon's interior. Lemma~\ref{lem:persistence} again
contradicts the assumption that $\lambda$ is on the boundary.
Therefore $\xi$ is not a listed power.
Finally, $N(0)$ and $N(1)$ have the same sign, and neither is zero.
Thus the origin lies strictly on the polygon's side of $\Lambda$.
\end{proof}

\begin{figure}[htbp]
\centering
\begin{tikzpicture}[x=3.5cm,y=3.5cm,font=\small]
\filldraw[fill=black!8,draw=black!50] (1.0000000,0.0000000) -- (0.1499968,0.9470427) -- (-0.8743908,0.2841068) -- (-0.4002171,-0.7854704) -- cycle;
\draw[thick] (-0.8601512,-0.7776828)--(0.7296209,-0.8046007);
\draw[thick,dashed] (-0.3865654,-1.0250818)--(-0.4934238,0.8504723);
\draw[black!55] plot coordinates {(-0.2911165,-0.7842009) (-0.2982985,-0.7844810) (-0.3054935,-0.7847332) (-0.3127016,-0.7849575) (-0.3199227,-0.7851538) (-0.3271570,-0.7853222) (-0.3344042,-0.7854627) (-0.3416645,-0.7855752) (-0.3489379,-0.7856599) (-0.3562243,-0.7857166) (-0.3635238,-0.7857453) (-0.3708364,-0.7857462) (-0.3781620,-0.7857191) (-0.3855006,-0.7856641) (-0.3928524,-0.7855812) (-0.4002171,-0.7854704) (-0.4075950,-0.7853316) (-0.4149859,-0.7851649) (-0.4223898,-0.7849703) (-0.4298068,-0.7847477) (-0.4372369,-0.7844972) (-0.4446800,-0.7842188) (-0.4521362,-0.7839125) (-0.4596054,-0.7835783) (-0.4670877,-0.7832161) (-0.4745831,-0.7828260) (-0.4820915,-0.7824080) (-0.4896129,-0.7819620) (-0.4971474,-0.7814881) (-0.5046950,-0.7809863) (-0.5122557,-0.7804566)};
\fill (1.0000000,0.0000000) circle[radius=.012] node[right] {$1$};
\fill (0.1499968,0.9470427) circle[radius=.012] node[above] {$\lambda$};
\fill (-0.8743908,0.2841068) circle[radius=.012] node[left] {$\lambda^2$};
\fill (-0.4002171,-0.7854704) circle[radius=.012] node[below left] {$\xi^2=\lambda^3$};
\draw[fill=white] (0.4905797,-0.8005533) circle[radius=.012] node[below] {$\xi$};
\draw[fill=white] (-0.4905797,0.8005533) circle[radius=.012] node[above left] {$-\xi$};
\node at (.31,.23) {$K_4(\lambda)$};
\node[below] at (0.4905797,-0.9005533) {supporting root};
\node[rotate=90,fill=white,inner sep=1pt] at (-0.4297958,-0.2663123) {opposite root};
\end{tikzpicture}
\caption{The two square roots of $\lambda^3$, for $q=4$ and
$\Arg\lambda=2\pi(9/40)$. The shaded polygon is
$K_4(\lambda)=\conv\{1,\lambda,\lambda^2,\lambda^3\}$.
The solid line through $\xi$ and $\xi^2=\lambda^3$ is tangent to the gray
curve $\Gamma_\lambda(s)$ and leaves the polygon on one side.
The dashed line through $-\xi$ and $\xi^2$ passes through the polygon's
interior. Both roots square to $\lambda^3$, so the squared equation alone
does not distinguish these two lines.}
\label{fig:root}
\end{figure}

Set $L=2q-j$, so $q<L\le2q$, and define the complex number $v$ by
\begin{equation}
 v=s_0-(s_0-1)z^q=z^q\overline\xi,\qquad
 v^2=z^L,\qquad
 1=\frac{s_0-1}{s_0}z^q+\frac1{s_0}v.\label{eq:bridge-chord}
\end{equation}
The last identity expresses $1$ as an average of $z^q$ and $v$ with two
positive weights summing to one. Thus $1$ lies strictly between these
points on their straight line segment.

Apply the map
\[
 T(w)=\frac{z^{2q}\overline w}{v}.
\]
It conjugates each point and then multiplies by a fixed nonzero complex
number, so it sends lines to lines and preserves convex combinations.
It sends $\xi$ to $z^q$, $\xi^2$ to $v$, and $K_q(\lambda)$ to
\begin{equation}
 C_z=\conv\{C_m=z^m/v:q<m\le2q\}.\label{eq:bridge-polygon}
\end{equation}
Since $T(0)=0$, Proposition~\ref{prop:bridge} puts $C_z$ and the origin on
the same side of the line through $z^q$ and $v$. This line also contains
$1$, by \eqref{eq:bridge-chord}.

The endpoints $z^q$ and $v$ lie on opposite sides of the real axis.
Choose $\eps=1$ when $z^q$ lies above it, and $\eps=-1$ when it lies below.
Integers $a,b$ then specify their angles as follows:
\begin{equation}
 \begin{gathered}
 z^q=R^qe^{i\eps\alpha},\qquad
 v=(-1)^bR^{L/2}e^{\pi iLx}=R^{L/2}e^{-i\eps\beta},\\
 \alpha=2\pi\eps(qx-a)>0,\qquad
 \beta=\pi\eps(b-Lx)>0,\qquad\alpha+\beta<\pi.
 \end{gathered}\label{eq:bridge-phases}
\end{equation}
The factor $(-1)^b$ records which square root of $z^L$ is used.
The angle sum is less than $\pi$ because the segment between the
endpoints contains the positive real point $1$ and its line does not
pass through the origin.

To record these integers, set $P=(q,a)$ and $Q=(L,b)$.
Their determinant is $qb-aL$, and the angle formulas give
\[
 L\alpha+2q\beta=2\pi\eps(qb-aL)>0.
\]
Thus the determinant has sign $\eps$.
Moreover, the coordinates $L,b$ cannot both be even:
\begin{equation}
 Q\notin2\Z^2.\label{eq:non-even}
\end{equation}
If both were even, the formula for $v$ would give $v=z^{L/2}$, and then
$\xi=\lambda^{q-L/2}$. Since $0\le q-L/2\le q-1$, this would make $\xi$
one of the listed powers, contradicting Proposition~\ref{prop:bridge}.

We now express the condition that every scaled power $C_m$ stays on the
origin's side of the line through $z^q$ and $v$ in terms of points with
integer coordinates.
For $P=(q,a)$ and $Q=(L,b)$, consider the region
\[
 H(P,Q)=[0,P]+[Q/2,Q]
 =\{\sigma P+\tau Q:0\le\sigma\le1,\ \tfrac12\le\tau\le1\}.
\]
The sum notation means that we add one point from each of the two
segments. The result is half of the parallelogram spanned by $P,Q$.
An integer point $U=(m,n)$ in this region selects the exponent $m$ in
$C_m=z^m/v$; only $q<m\le2q$ corresponds to one of the listed powers.
The second coordinate $n$ records a whole-turn change of angle.

Write $U=\sigma P+\tau Q$. Its coordinates give
\[
 \log|C_m|=\sigma q\log R+(2\tau-1)\frac L2\log R,\qquad
 \arg C_m\equiv\eps\bigl(\sigma\alpha-(2\tau-1)\beta\bigr)
 \pmod{2\pi}.
\]
Let $\psi(\varphi)$ be the logarithm of the radius where the ray at angle
$\varphi$ meets the line through $z^q$ and $v$.
Since $\sigma,2\tau-1\in[0,1]$, Lemma~\ref{lem:convex} gives
\[
 \psi\bigl(\eps(\sigma\alpha-(2\tau-1)\beta)\bigr)
 \le\sigma q\log R+(2\tau-1)\frac L2\log R=\log|C_m|.
\]
Thus $C_m$ is on or beyond that line on its ray.
Equality can occur only at $U=Q/2$, $P+Q/2$ or $Q$.
The first point has $m=L/2\le q$ and is outside the required exponent
range. The second is not an integer point, because $L,b$ are not both
even. The last is the already selected endpoint $Q$.
Therefore any other integer point in $H(P,Q)$ with $q<m\le2q$ would put
$C_m$ strictly on the side opposite the origin, contradicting the
established position of $C_z$.
The next lemma shows how the absence of such a point restricts $P,Q$.
These restrictions will identify the neighboring Farey fractions when
we select the lateral edge.

\begin{lemma}\label{lem:lattice}
Let $P=(q,a)$ and $Q=(L,b)$ have integer coordinates, with
$q<L\le2q$, and suppose $L,b$ are not both even.
Assume $\det(P,Q)>0$.
Write
\[
 d=\gcd(q,a),\qquad r=q/d,\qquad p=a/d,\qquad P=d(r,p).
\]
Thus $(r,p)$ is primitive: its coordinates have no common divisor greater
than $1$. Put
\[
 \delta=\det((r,p),Q)=rb-pL>0.
\]
Exactly one of the following holds:
\begin{enumerate}
\item The region $H(P,Q)=[0,P]+[Q/2,Q]$ contains an integer point
$U=(m,n)$ distinct from $Q$ with $q<m\le2q$.
\item $\delta=1$ and $r+L>2q$. In this case, $Q$ is the only integer point
in $H(P,Q)$ whose first coordinate belongs to $(q,2q]$.
\end{enumerate}
For negative determinant, replace every point $(m,n)$ by $(m,-n)$;
in particular, replace $P,Q$ by $(q,-a),(L,-b)$.
This reverses the determinant's sign and preserves the restriction on
the first coordinate, so the positive-determinant statement applies.
\end{lemma}

Appendix~\ref{app:lattice} gives the complete proof of
Lemma~\ref{lem:lattice}, grouping integer points $U$ by the value of
$\det((r,p),U)$.
Figure~\ref{fig:lattice} illustrates the two alternatives: whether
$H(P,Q)$ contains another integer point besides $Q$ whose first
coordinate satisfies $q<m\le2q$.
\begin{figure}[htbp]
\centering
\begin{tikzpicture}[x=.48cm,y=.80cm,>=Stealth,font=\small]
\begin{scope}
\path[use as bounding box] (1.5,-1.0) rectangle (10.0,3.65);
\fill[black!6] (4,0) rectangle (8,3.35);
\draw[black!45,dashed] (4,0)--(4,3.35);
\draw[black!45,dashed] (8,0)--(8,3.35);
\filldraw[fill=black!15,draw=black,thick] (2.5,1.5)--(6.5,1.5)--(9,3)--(5,3)--cycle;
\draw[->,black!50] (1.7,0)--(9.6,0) node[right] {$m$};
\draw[->,black!50] (1.9,0)--(1.9,3.45) node[above] {$\nu$};
\fill (5,3) circle[radius=2pt] node[above left] {$Q$};
\fill (6,2) circle[radius=2pt] node[right] {$U$};
\node[below left] at (2.5,1.5) {$Q/2$};
\node[below right] at (6.5,1.5) {$P+Q/2$};
\node[below] at (4,0) {$q$};\node[below] at (8,0) {$2q$};
\node at (5.7,-.86) {$Q=(5,2)$: an additional point};
\end{scope}
\begin{scope}[xshift=6.7cm]
\path[use as bounding box] (2.0,-1.0) rectangle (11.7,3.65);
\fill[black!6] (4,0) rectangle (8,2.0);
\draw[black!45,dashed] (4,0)--(4,2.0);
\draw[black!45,dashed] (8,0)--(8,2.0);
\filldraw[fill=black!15,draw=black,thick] (3.5,.5)--(7.5,.5)--(11,1)--(7,1)--cycle;
\draw[->,black!50] (2.2,0)--(11.3,0) node[right] {$m$};
\draw[->,black!50] (2.4,0)--(2.4,2.35) node[above] {$\nu$};
\fill (7,1) circle[radius=2pt] node[above left] {$Q$};
\node[above left] at (3.5,.5) {$Q/2$};
\draw[black!50,thin] (7.5,.5)--(9.7,1.65);
\node[above] at (9.7,1.65) {$P+Q/2$};
\node[above] at (4,0) {$q$};\node[below] at (8,0) {$2q$};
\node at (6.8,-.86) {$Q=(7,2)$: determinant one};
\end{scope}
\end{tikzpicture}
\caption{The two alternatives in Lemma~\ref{lem:lattice}, with $q=4$ and
$P=(4,1)$. The light gray band marks the range $4<m\le8$; the darker
parallelogram is $H(P,Q)=[0,P]+[Q/2,Q]$.
A labeled integer point $(m,n)$ is drawn at $(m,\nu)$, where
$\nu=\det(P,(m,n))=4n-m$.
On the left, for $Q=(5,2)$, the point $U=(6,2)$ also lies in $H(P,Q)$
and has $4<m\le8$.
On the right, for $Q=(7,2)$, $Q$ is the only integer point in $H(P,Q)$
with $4<m\le8$. This is the second case of the lemma, since
$\det(P,Q)=1$ and $4+7>8$.
Only the points needed for these comparisons are marked.}
\label{fig:lattice}
\end{figure}

\begin{proposition}\label{prop:lateral}
Let $\lambda$ be a nonreal boundary point of $\Sigma_q$ with
$|\lambda|<1$, and suppose $x=\Arg\lambda/(2\pi)$ lies strictly between
consecutive fractions in $\F_q$.
Let $s_0=s_{\max}$ be the largest parameter for which
$\Gamma_\lambda(s_0)\in K_q(\lambda)$, as in
Proposition~\ref{prop:alternatives}.
Suppose that at this last intersection the curve reaches one of the
points defining the convex hull:
\[
 \Gamma_\lambda(s_0)=\lambda^j,\qquad 0\le j\le q-1,\qquad
 1<s_0<s^*(\lambda).
\]
This is the condition in \eqref{eq:one-site}.
Put $\xi=1+s_0(\lambda^q-1)$, so $\xi^2=\lambda^j$, and set $L=2q-j$.
One endpoint fraction has denominator dividing $q$, so its associated
matrix is $U_q$. The other endpoint is a reduced fraction $b/L$, with
$q<L<2q$, and its associated matrix is $U_L$.
For $\lambda=\rho e^{2\pi ix}$, the square root associated with the last
intersection is the root in \eqref{eq:root}, namely
\[
 \xi=(-1)^b\rho^{j/2}e^{\pi ijx}.
\]
With $c=1-1/s_0\in(0,1)$, equation \eqref{eq:lateral-equation} reads
\[
 \lambda^q=c+(1-c)\xi.
\]
It expresses $\lambda^q$ as a weighted average of the complex numbers
$1$ and the chosen square root $\xi$.
The matrix $cU_q+(1-c)U_L$ has $\lambda$ as an eigenvalue.
\end{proposition}

\begin{proof}
Apply Lemma~\ref{lem:lattice}, negating the second coordinates when
$\eps=-1$ as described in its statement.
Suppose its first alternative gives an integer point
\[
 U=(m,n)=\sigma P+\tau Q\ne Q,\qquad
 0\le\sigma\le1,\quad\tfrac12\le\tau\le1,\quad q<m\le2q.
\]
Set $a_0=\sigma$ and $b_0=2\tau-1$, so both belong to $[0,1]$.
The identity $2U-Q=2a_0P+b_0Q$ gives
\[
 \begin{aligned}
 \log|C_m|&=a_0q\log R+b_0(L/2)\log R,\\
 \arg C_m&\equiv\eps(a_0\alpha-b_0\beta)\pmod{2\pi}.
 \end{aligned}
\]
Subtracting the displayed angle from $2\pi mx-\pi(Lx-b)$ gives
$2\pi n$, an integer multiple of a full turn. The displayed angle lies
between the endpoint angles $\eps\alpha$ and $-\eps\beta$.
Let $\psi$ be the logarithm of the intersection radius for the line
through $z^q$ and $v$. Lemma~\ref{lem:convex}, with the angle signs given
by $\eps$, yields
\[
 \psi\bigl(\eps(a_0\alpha-b_0\beta)\bigr)
 \le a_0q\log R+b_0(L/2)\log R=\log|C_m|.
\]
Equality is possible only at $U=Q/2$, $P+Q/2$ or $Q$.
The first point has $m=L/2\le q$, the second is not an integer point by
\eqref{eq:non-even}, and the third is excluded by $U\ne Q$.
The inequality is therefore strict: $C_m$ lies beyond the line on its
ray, on the side opposite the origin.
This contradicts the established position of the polygon $C_z$ in
\eqref{eq:bridge-polygon}.

Thus the second alternative of Lemma~\ref{lem:lattice} holds.
Writing $P=d(r,p)$ and returning to the original coordinates gives
\[
 \eps(rb-pL)=1,\qquad r+L>2q.
\]
The fraction $p/r=a/q$ is reduced by the definition of $(r,p)$.
The fraction $b/L$ is also reduced, because a common divisor of $b,L$
would divide $rb-pL$, whose absolute value is $1$.
The angle inequalities \eqref{eq:bridge-phases} give
$p/r<x<b/L$ when $\eps=1$, and $b/L<x<p/r$ when $\eps=-1$.

No reduced fraction of denominator at most $2q$ can lie between these
two fractions.
Indeed, in the ordering $p/r<b/L$, any fraction $y/t$ between them has
$ry-pt\ge1$ and $bt-Ly\ge1$, so
\[
 t=L(ry-pt)+r(bt-Ly)\ge L+r>2q.
\]
The opposite ordering gives the same bound after interchanging the
endpoints. Thus the two fractions are consecutive in $F_{2q}$.
Neither can lie outside $[1/(2q),1/2]$: otherwise one endpoint of this
interval would lie between them with denominator at most $2q$.
They are therefore exactly the two endpoints of the cell containing $x$
in $\F_q$.

Since $r$ divides $q$, the smallest multiple of $r$ that is at least $q$
is $q$ itself. Since $L>q$, the corresponding multiple for denominator
$L$ is $L$. Thus these fractions select the vertices $U_q$ and $U_L$.
Also $L<2q$. If $L=2q$, then $r\mid q$ makes $r$ a divisor of both terms
in $rb-2qp$, so $|rb-2qp|=1$ forces $r=1$.
The fraction $p/r$ would then be an integer, which cannot belong to
$[1/(2q),1/2]$.

Finally, write $\rho=|\lambda|=R^{-1}$.
Conjugating $\overline\xi=z^{-q}v$, with the expression for $v$ in
\eqref{eq:bridge-phases}, gives
\[
 \xi=(-1)^b\rho^{(2q-L)/2}e^{\pi i(2q-L)x}.
\]
This is the root in \eqref{eq:root}, with $j=2q-L$.
Rearranging $\xi=1+s_0(\lambda^q-1)$ gives
\[
 \lambda^q-c=(1-c)\xi,\qquad c=1-1/s_0\in(0,1),
\]
This is equation \eqref{eq:lateral-equation}.
Since $\xi^2=\lambda^{2q-L}$, squaring the equality gives
\[
 (\lambda^q-c)^2-(1-c)^2\lambda^{2q-L}=0.
\]
This is the characteristic equation of
$B(c,c,e_{2q-L})=cU_q+(1-c)U_L$, so that matrix has $\lambda$ as an
eigenvalue.
\end{proof}

\begin{corollary}\label{cor:exhaustion}
Let $\lambda=\rho e^{2\pi ix}$ be a boundary point of $\Sigma_q$ with
$0<\rho<1$, and suppose $x$ lies strictly between consecutive fractions
$f=h/k<g=r/s$ in $\F_q$. Put $M=\ell_q(k)$ and $N=\ell_q(s)$, as in \eqref{eq:edge}.
Exactly one of the following holds:
\begin{enumerate}
\item $M,N>q$, and for some $t\in(0,1)$, equation
\eqref{eq:base-equation} holds:
\[
 \lambda^{2q}=(1-t)\lambda^{2q-M}+t\lambda^{2q-N}.
\]
The matrix $(1-t)U_M+tU_N$ has $\lambda$ as an eigenvalue.
\item One of $M,N$ equals $q$ and the other equals $L$, with $q<L<2q$.
Write the endpoint fraction associated with $U_L$ as $\ell/L$.
For some $c\in(0,1)$, equations \eqref{eq:root} and
\eqref{eq:lateral-equation} give
\[
 \xi=(-1)^\ell\rho^{(2q-L)/2}e^{\pi i(2q-L)x},\qquad
 \lambda^q=c+(1-c)\xi.
\]
The matrix $cU_q+(1-c)U_L$ has $\lambda$ as an eigenvalue.
\end{enumerate}
These two equations, with the square root specified in the second case,
account for every possible last intersection
$\Gamma_\lambda(s_{\max})\in K_q(\lambda)$ at such a boundary point.
\end{corollary}

\begin{proof}
Proposition~\ref{prop:alternatives} shows that the last intersection is
either at $s_{\max}=1$, strictly between two of the powers
$1,\lambda,\ldots,\lambda^{q-1}$, or at $s_{\max}>1$, equal to one of them.

When $s_{\max}=1$, Proposition~\ref{prop:base} says that the two endpoints
are $\lambda^{2q-M}$ and $\lambda^{2q-N}$.
Since $\Gamma_\lambda(1)=\lambda^{2q}$, there is a $t\in(0,1)$ such that
\[
 \lambda^{2q}=(1-t)\lambda^{2q-M}+t\lambda^{2q-N}.
\]
The same coefficients $1-t$ and $t$ give the matrix
$(1-t)U_M+tU_N$. The displayed equality is its characteristic equation
\eqref{eq:base-equation}, so $\lambda$ is an eigenvalue of this matrix.

When $s_{\max}>1$, Proposition~\ref{prop:lateral} gives one endpoint
matrix $U_q$ and the other $U_L$, with $q<L<2q$.
It also specifies the square root $\xi$ in \eqref{eq:root} and proves
$\lambda^q=c+(1-c)\xi$, where $c=1-1/s_{\max}\in(0,1)$.
Thus $cU_q+(1-c)U_L$ has $\lambda$ as an eigenvalue.

The two alternatives cannot overlap.
If both $M$ and $N$ were $q$, then both denominators $k,s$ would divide
$q$, so $k+s\le2q$. But consecutive fractions in $\F_q$ satisfy
$k+s>2q$ by \eqref{eq:farey-data}.
Hence the two endpoint matrices cannot both be $U_q$.
\end{proof}

\section{The boundary radius and smaller attainable radii}\label{sec:filling}
Corollary~\ref{cor:exhaustion} restricts every nonreal boundary eigenvalue
whose angle lies strictly between consecutive Farey fractions to one of
the two equations stated there. We now prove that the formulas in
Section~1 produce an attainable eigenvalue $r_q(x)e^{2\pi ix}$ at each
such angle. For $q\ge3$, we then show that no attainable eigenvalue on
the same ray has a larger radius, and that every smaller radius is
attainable. The last step uses Lemma~\ref{lem:persistence}, which
approximates nonreal attainable eigenvalues inside the unit disk by
eigenvalues in the interior of $\Sigma_q$.
\begin{proposition}\label{prop:carrier}
Let $f=h/k<g=r/s$ be consecutive fractions in $\F_q$, and put
$M=\ell_q(k)$ and $N=\ell_q(s)$ as in \eqref{eq:edge}.
For every $x\in(f,g)$, the formulas \eqref{eq:base-data} or
\eqref{eq:lateral-data}, as applicable, determine a unique radius
$\rho=r_q(x)\in(0,1)$. The function $x\mapsto r_q(x)$ is real-analytic
on $(f,g)$, and $\lambda=r_q(x)e^{2\pi ix}$ is an eigenvalue of the
following matrix:
\begin{enumerate}
\item If $M,N>q$, the coefficient $t$ in \eqref{eq:base-t} lies in
$(0,1)$, and the matrix is $(1-t)U_M+tU_N$.
\item If one of $M,N$ equals $q$ and the other equals $L$, the
coefficient $c$ in \eqref{eq:lateral-equation}, computed with the
square root in \eqref{eq:root}, lies in $(0,1)$, and the matrix is
$cU_q+(1-c)U_L$.
\end{enumerate}
\end{proposition}
\begin{proof}
Suppose first that $M,N>q$. For $q\ge3$, Lemma~\ref{lem:phases}
implies that the angles in \eqref{eq:base-data} satisfy
$\alpha,\beta>0$ and $\alpha+\beta<\pi$.
For $q=2$, this case occurs only on $(1/4,1/3)$; there
$\alpha=2\pi(4x-1)$ and $\beta=2\pi(1-3x)$, so both are positive and
$\alpha+\beta=2\pi x<\pi$.

Lemma~\ref{lem:chord}, with exponents $M,N$, gives a unique
$\rho\in(0,1)$. Put $\lambda=\rho e^{2\pi ix}$ and
$z=1/\overline\lambda$. The same lemma gives
\[
 1=(1-t)z^M+tz^N,\qquad
 t=\frac{\rho^N\sin\alpha}{\sin(\alpha+\beta)}\in(0,1),
\]
which is the coefficient in \eqref{eq:base-t}.
Conjugating and multiplying by $\lambda^{2q}$ yields
\[
 \lambda^{2q}=(1-t)\lambda^{2q-M}+t\lambda^{2q-N}.
\]
This is the characteristic equation \eqref{eq:base-equation} of
$(1-t)U_M+tU_N$, so this matrix has $\lambda$ as an eigenvalue.

Now suppose one endpoint matrix is $U_q$ and the other is $U_L$.
Denote their endpoint fractions by $h/k$ and $\ell/L$, respectively.
Then $k\mid q$, $k+L>2q$, $|k\ell-hL|=1$ and $q<L<2q$.
The angles $\alpha,\beta$ defined in \eqref{eq:lateral-data} are positive
inside the interval. Their sum is affine in $x$, and its values at
$h/k$ and $\ell/L$, respectively, are
\[
 \frac\pi k,\qquad \frac{2\pi q}{kL}.
\]
Both values are less than $\pi$, since $k\ge2$ and $kL>2q$.
Thus $\alpha+\beta<\pi$ throughout the interval, and
Lemma~\ref{lem:chord} applies with exponents $q,L/2$ to give a unique
$\rho\in(0,1)$.

Again put $\lambda=\rho e^{2\pi ix}$ and $z=1/\overline\lambda$, and set
\[
 v=(-1)^\ell\rho^{-L/2}e^{\pi iLx},\qquad v^2=z^L.
\]
The definitions give
$z^q=\rho^{-q}e^{i\eps\alpha}$ and
$v=\rho^{-L/2}e^{-i\eps\beta}$.
Lemma~\ref{lem:chord}, conjugating its identity if $\eps=-1$, therefore gives
\[
 1=cz^q+(1-c)v,\qquad 0<c<1.
\]
Multiplying by $z^{-q}$ and conjugating yields
\[
 \lambda^q=c+(1-c)\xi,\qquad
 \xi=\overline{z^{-q}v}
     =(-1)^\ell\rho^{(2q-L)/2}e^{\pi i(2q-L)x}.
\]
These are precisely \eqref{eq:lateral-equation} and the square root
specified in \eqref{eq:root}.
Since $\xi^2=\lambda^{2q-L}$, squaring gives
$(\lambda^q-c)^2=(1-c)^2\lambda^{2q-L}$, the characteristic equation
of $cU_q+(1-c)U_L$. Hence this matrix has $\lambda$ as an eigenvalue.
Moreover, with $s=(1-c)^{-1}>1$, the same identity gives
$\xi=1+s(\lambda^q-1)$.
Because $|\xi|=\rho^{(2q-L)/2}<1$, the bound in \eqref{eq:sstar}
is strict: $1<s<s^*(\lambda)$.

Finally, in \eqref{eq:scalar} use $(u,v)=(M,N)$ in the first case
and $(u,v)=(q,L/2)$ in the second. The derivative of its left-hand side is
\[
 v\rho^{v-1}\sin\alpha+u\rho^{u-1}\sin\beta>0.
\]
The angles depend affinely on $x$, so the implicit-function theorem
gives real-analytic dependence of $\rho=r_q(x)$ on $x$ in both cases.
\end{proof}
\begin{proposition}\label{prop:outermost}
Let $q\ge3$, let $f<g$ be consecutive fractions in $\F_q$, and fix
$x\in(f,g)$. If $\lambda=\rho e^{2\pi ix}\in\Sigma_q$ with $\rho>0$, then
\[
 \rho\le r_q(x).
\]
Thus $r_q(x)$ is the largest attainable modulus at angle $2\pi x$.
\end{proposition}
\begin{proof}
By Proposition~\ref{prop:carrier}, the point $r_q(x)e^{2\pi ix}$ belongs
to $\Sigma_q$. Since $\Sigma_q$ is compact, the set
\[
 A_x=\{\rho\in[0,1]:\rho e^{2\pi ix}\in\Sigma_q\}
\]
has a largest element $\rho_{\max}\ge r_q(x)>0$.
Put $\lambda_{\max}=\rho_{\max}e^{2\pi ix}$.
This point lies on $\partial\Sigma_q$: if it were an interior point,
a small increase in its modulus with the same angle would still give
a point of $\Sigma_q$, contradicting the definition of $\rho_{\max}$.

Moreover, $\rho_{\max}<1$.
By \eqref{eq:unit}, an attainable eigenvalue of modulus $1$ is a root
of unity of order at most $2q$. Its normalized angle $x$ would therefore
be a fraction in $\F_q$, whereas $x$ lies strictly between consecutive
fractions in that list.

Corollary~\ref{cor:exhaustion} now applies to $\lambda_{\max}$.
Let $M,N$ be the lengths associated with $f,g$ in \eqref{eq:edge}.
If $M,N>q$, then $\lambda_{\max}$ satisfies \eqref{eq:base-equation}
for the matrix $(1-t)U_M+tU_N$, with $0<t<1$.
Otherwise, one length equals $q$ and the other equals $L$;
$\lambda_{\max}$ satisfies \eqref{eq:lateral-equation} with the square
root in \eqref{eq:root}, for the matrix $cU_q+(1-c)U_L$ with $0<c<1$.
Proposition~\ref{prop:carrier} gives a unique radius for the applicable
equation at angle $2\pi x$, so $\rho_{\max}=r_q(x)$.
Every attainable eigenvalue on this ray therefore has modulus at most
$r_q(x)$.
\end{proof}
\begin{theorem}[Attainability of smaller radii]\label{thm:filling}
Let $q\ge3$, and let $f<g$ be consecutive fractions in $\F_q$.
Then the region
\begin{equation}
 D_{f,g}=\{\rho e^{2\pi ix}:f<x<g,\quad0<\rho<r_q(x)\}\label{eq:cell-region}
\end{equation}
is contained in $\interior\Sigma_q$.
\end{theorem}
\begin{proof}
By Proposition~\ref{prop:carrier}, $r_q$ is positive and continuous on
$(f,g)$, so $D_{f,g}$ is open and connected.
We first show that it contains an attainable eigenvalue.
Choose $x_0\in(f,g)$ and put
$\lambda_0=r_q(x_0)e^{2\pi ix_0}\in\Sigma_q$, again using
Proposition~\ref{prop:carrier}.
This point is nonreal and has modulus less than one, so
Lemma~\ref{lem:persistence} gives points
$\lambda_n\in\interior\Sigma_q$ with $\lambda_n\to\lambda_0$.
For sufficiently large $n$, write
$\lambda_n=\rho_n e^{2\pi ix_n}$ with $x_n\in(f,g)$ and $\rho_n>0$.
Proposition~\ref{prop:outermost} gives $\rho_n\le r_q(x_n)$.
Equality is impossible: an open neighborhood of $\lambda_n$ in
$\Sigma_q$ would contain an eigenvalue of the same angle and a larger
modulus, contradicting that proposition.
Thus $0<\rho_n<r_q(x_n)$, and $D_{f,g}\cap\Sigma_q$ is nonempty.

Next, $D_{f,g}$ contains no point of $\partial\Sigma_q$.
Indeed, Corollary~\ref{cor:exhaustion} requires any such boundary point
to satisfy the applicable equation \eqref{eq:base-equation} or
\eqref{eq:lateral-equation}, with the specified square root
\eqref{eq:root} in the latter case.
Proposition~\ref{prop:carrier} identifies its radius as $r_q(x)$,
whereas every point of $D_{f,g}$ has radius strictly smaller than this.
Hence
\[
 D_{f,g}\cap\partial\Sigma_q=\varnothing.
\]

Set $A=D_{f,g}\cap\Sigma_q$.
This nonempty set is closed relative to $D_{f,g}$ because $\Sigma_q$
is closed. It is also open relative to $D_{f,g}$: every point of $A$
belongs to $\Sigma_q$ but not its boundary, and therefore has an open
neighborhood contained in $\Sigma_q$.
If $A$ were not all of $D_{f,g}$, the two nonempty sets $A$ and
$D_{f,g}\setminus A$ would separate $D_{f,g}$ into disjoint relatively
open sets, contradicting connectedness.
Thus $D_{f,g}\subseteq\Sigma_q$.
Since $D_{f,g}$ is itself open, it lies in $\interior\Sigma_q$.
\end{proof}
Corollary~\ref{cor:exhaustion} identifies the equations that boundary
eigenvalues must satisfy, and Proposition~\ref{prop:carrier} gives their
unique radius for each $x\in(f,g)$. Together they yield the bound
$|\lambda|\le r_q(x)$ in Proposition~\ref{prop:outermost}.
Theorem~\ref{thm:filling} then uses this bound to place the interior
approximations from Lemma~\ref{lem:persistence} at radii strictly smaller
than $r_q(x)$; connectedness proves that all smaller positive radii are
attainable. The approximation lemma applies even when the boundary
eigenvalue is a repeated root of a realizing matrix's characteristic
polynomial, and no direction of approach needs to be chosen.

\section{Endpoint limits, the complete boundary and realizing matrices}\label{sec:endpoints}
\subsection{Limits at the endpoints}
\begin{proposition}\label{prop:limits}
As $x\downarrow1/(2q)$, the radius $r_q(x)$ tends to $1$.
At every fraction $f\in\F_q$ with $1/(2q)<f<1/2$, both the left-hand
and right-hand limits of $r_q(x)$ are $1$.
At $x=1/2$, which corresponds to the negative real axis,
\[
 \lim_{x\uparrow1/2}r_q(x)=
 \begin{cases}
  1,&q\text{ even},\\
  \rho_q,&q\text{ odd},
 \end{cases}
\]
where, for odd $q$, $\rho_q\in(0,1)$ is the unique solution of
\eqref{eq:odd-limit}, namely
\[
 (2q-1)\rho_q^{q+1}+(q+1)\rho_q^{2q-1}=q-2.
\]
\end{proposition}
\begin{proof}
First consider an interval with $M=uk>q$ and $N=vs>q$.
By \eqref{eq:base-data}, as $x$ approaches $h/k$, the angle $\alpha$
tends to zero and $\beta$ tends to $2\pi v/k$.
As $x$ approaches $r/s$, $\beta$ tends to zero and $\alpha$ tends to
$2\pi u/s$. For $q\ge3$, Lemma~\ref{lem:phases} puts the remaining
angle strictly between $0$ and $\pi$, except at $x=1/2$ when $q$ is odd,
which we treat below. For $q=2$, the only such interval is $(1/4,1/3)$,
and the same conclusion follows directly from the angle formulas
in the proof of Proposition~\ref{prop:carrier}.

If one endpoint matrix is $U_q$ and the other is $U_L$, the angles
in \eqref{eq:lateral-data} have the same behavior: one tends to zero,
and the other tends to $\pi/k$ at the endpoint $h/k$, or to
$2\pi q/(kL)$ at the endpoint $\ell/L$.
Both limits lie in $(0,\pi)$, as shown in the proof of
Proposition~\ref{prop:carrier}.

Whenever one angle tends to zero and the other to
$\gamma\in(0,\pi)$, the limiting form of \eqref{eq:scalar} is
\[
 \rho^w\sin\gamma=\sin\gamma,
\]
where $w$ is the exponent of the term whose sine factor does not vanish:
$M$ or $N$ in the first case, and $q$ or $L/2$ in the second.
To prove convergence without assuming that the radius has a limit,
fix $r\in(0,1)$. The left-hand side minus the right-hand side of
\eqref{eq:scalar}, evaluated at $\rho=r$, tends to
\[
 (r^w-1)\sin\gamma<0.
\]
Since that difference is strictly increasing in $\rho$ and vanishes
at $\rho=r_q(x)$, we have $r<r_q(x)<1$ near the endpoint.
As $r$ can be arbitrarily close to $1$, the radius tends to $1$.
This proves the limit at $1/(2q)$ and both limits at every internal
fraction of $\F_q$.

The fraction immediately preceding $1/2$ in $F_{2q}$ is
$(q-1)/(2q-1)$.
If $q$ is even, the endpoint $1/2$ corresponds to $U_q$, since
$\ell_q(2)=q$. The preceding argument therefore also gives the limit
$1$ at this endpoint.
If $q$ is odd, the two matrices for the final interval are
$U_{2q-1}$ and $U_{q+1}$, so $M=2q-1$ and $N=q+1$.
Put $x=1/2-\eps$, with $\eps>0$. Then \eqref{eq:base-data} gives
\[
 \alpha=\pi-2\pi(2q-1)\eps,\qquad
 \beta=2\pi(q+1)\eps,
\]
and hence $\alpha+\beta=\pi-2\pi(q-2)\eps$.
Dividing \eqref{eq:scalar} by $2\pi\eps$ and letting $\eps\downarrow0$
gives
\[
 (2q-1)\rho^{q+1}+(q+1)\rho^{2q-1}=q-2,
\]
using $\sin t/t\to1$ as $t\to0$.
This is \eqref{eq:odd-limit}.
Its left-hand side is strictly increasing from $0$ to $3q$ as
$\rho$ increases from $0$ to $1$.
Since $0<q-2<3q$, there is exactly one solution $\rho_q\in(0,1)$.

To verify convergence to this solution, choose fixed radii
$0<r_-<\rho_q<r_+<1$.
In the divided equation, the left-hand side minus the right-hand side
is negative at $r_-$ and positive at $r_+$ for sufficiently small
$\eps$. Its strict increase in $\rho$ gives
$r_-<r_q(1/2-\eps)<r_+$.
Taking $r_-$ and $r_+$ arbitrarily close to $\rho_q$ proves
$r_q(x)\to\rho_q$ as $x\uparrow1/2$.
\end{proof}

\subsection{Completing the spectral region and its boundary}
\begin{proof}[Proof of Theorem~\ref{thm:main} for $q\ge3$]
For $x$ strictly between consecutive fractions in $\F_q$,
Theorem~\ref{thm:filling} makes every point
$\rho e^{2\pi ix}$ with $0<\rho<r_q(x)$ an interior point of $\Sigma_q$.
Proposition~\ref{prop:carrier} supplies the point at radius $r_q(x)$,
and Proposition~\ref{prop:outermost} excludes every larger radius.
It remains to treat the endpoint angles and identify the boundary.

Let $a\in\F_q$ satisfy $1/(2q)<a<1/2$.
Thus $a$ is neither the first nor the last fraction in the list.
Such fractions are called internal fractions.
Proposition~\ref{prop:limits} and closedness of $\Sigma_q$ include the
whole segment $\{\rho e^{2\pi ia}:0\le\rho\le1\}$.
To show that its points with $0<\rho<1$ are interior, fix such a $\rho$
and choose $0<\rho_-<\rho<\rho_+<1$.
Both adjacent limits of $r_q(x)$ are $1$, so for sufficiently small
$\delta>0$, $r_q(x)>\rho_+$ whenever $0<|x-a|<\delta$.
Theorem~\ref{thm:filling} therefore includes every point of
\[
 V=\{r e^{2\pi ix}:\rho_-<r<\rho_+,\quad a-\delta<x<a+\delta\}
\]
except possibly those with $x=a$.
Closedness includes those points as limits of the others.
Thus the open set $V$ is contained in $\Sigma_q$, proving that
$\rho e^{2\pi ia}$ is interior.
The unit endpoint is boundary because $\Sigma_q$ lies in the unit disk.

At the negative real axis, use the final upper interval and its complex
conjugate. The same neighborhood argument, with the limits in
Proposition~\ref{prop:limits}, shows that $(-1,0)$ is interior for even
$q$, and $(-\rho_q,0)$ is interior for odd $q$.
Proposition~\ref{prop:elementary} gives the whole real interval $[-1,1]$.
For odd $q$ and $\rho_q<r\le1$, the limit $r_q(x)\to\rho_q$ as
$x\uparrow1/2$ gives $r>r_q(x)$ at angles sufficiently close to $\pi$.
By Proposition~\ref{prop:outermost}, the points $r e^{2\pi ix}$ at these
angles lie outside $\Sigma_q$ and approach the attainable point $-r$.
Hence every such $-r$ is boundary; closedness of $\partial\Sigma_q$
also includes $-\rho_q$. Thus $[-1,-\rho_q]$ is boundary for odd $q$.
For even $q$, $-1$ is boundary because it lies on the unit circle.

Proposition~\ref{prop:elementary} places the segments
$[0,1]$, $[0,e^{i\pi/q}]$ and $[0,e^{-i\pi/q}]$ on the boundary,
including the origin, and excludes nonzero eigenvalues with
$0<|\Arg\lambda|<\pi/q$.
Every point $r_q(x)e^{2\pi ix}$ with $x$ strictly between consecutive
Farey fractions is boundary: it is attainable, but increasing its
modulus gives points outside $\Sigma_q$ by
Proposition~\ref{prop:outermost}.
Together with symmetry under complex conjugation and the endpoint
analysis above, this proves \eqref{eq:region} and \eqref{eq:boundary}.
The unit-circle classification \eqref{eq:unit} was already proved in
Proposition~\ref{prop:elementary}.

Theorem~\ref{thm:triangles} realizes every attainable eigenvalue by a
weighted average of $U_q$ and at most two other vertices.
Proposition~\ref{prop:carrier} gives averages of two endpoint matrices
for boundary angles strictly between consecutive Farey fractions;
the next subsection supplies matrices for the remaining boundary points.
Appendix~\ref{app:four} proves the case $q=2$ directly.
The interior approximation used here requires $q\ge3$.
\end{proof}

\subsection{Explicit realizing matrices}
Let $\lambda=\rho e^{2\pi ix}$ be a nonreal boundary eigenvalue with
$x$ strictly between consecutive fractions $f<g$ in $\F_q$.
Let $M,N$ be their associated lengths in \eqref{eq:edge}.
If $M,N>q$, use $(1-t)U_M+tU_N$, with $t$ given by \eqref{eq:base-t}.
If one length is $q$ and the other is $L$, use $cU_q+(1-c)U_L$, with
$c$ given by \eqref{eq:lateral-equation} using the root in \eqref{eq:root}.
At a Farey fraction $h/k$, the unit-circle point $e^{2\pi ih/k}$ is
an eigenvalue of $U_{\ell_q(k)}$.

For $0\le c\le1$, the matrix $cU_q+(1-c)U_{2q}$ has characteristic
polynomial \eqref{eq:first-edge}, namely
\[
 (X^q-1)\bigl(X^q-(2c-1)\bigr).
\]
Consequently, every $\lambda$ with $\lambda^q=2c-1$ is an eigenvalue
of this matrix. For $r\in[0,1]$, choose
\[
 \begin{aligned}
 c&=\frac{1+r^q}{2}&&\text{to realize }\lambda=r,\\
 c&=\frac{1-r^q}{2}&&\text{to realize }\lambda=re^{\pm i\pi/q},\\
 c&=\frac{1+(-r)^q}{2}&&\text{to realize }\lambda=-r.
 \end{aligned}
\]
Each choice lies in $[0,1]$ and gives the required eigenvalue through
$\lambda^q=2c-1$.

For odd $q$, every point of $[-1,0]$ can also be realized by a weighted
average of $U_{2q-1}$ and $U_{q+1}$.
For $0\le r\le1$, set
\begin{equation}
 t(r)=\frac{1+r^{2q-1}}{1+r^{q-2}}\in[0,1].\label{eq:negative-t}
\end{equation}
The stated range follows from $r^{2q-1}\le r^{q-2}$.
With this choice, $\lambda=-r$ satisfies
\[
 \lambda^{2q}=(1-t(r))\lambda+t(r)\lambda^{q-1},
\]
which is the characteristic equation of
$(1-t(r))U_{2q-1}+t(r)U_{q+1}$.
Thus this matrix has $-r$ as an eigenvalue.
In particular, every point of the boundary interval $[-1,-\rho_q]$
is realized using the same two matrices as the curved boundary at
angles $2\pi x$ with $(q-1)/(2q-1)<x<1/2$.

\subsection{A double eigenvalue for odd $q$}\label{sec:collision}
For odd $q$, the matrix $(1-t)U_{2q-1}+tU_{q+1}$ has characteristic
polynomial
\begin{equation}
 F_q(X,t)=X^{2q}-(1-t)X-tX^{q-1},\qquad 0\le t\le1.
 \label{eq:terminal-polynomial}
\end{equation}
We show that the boundary endpoint $-\rho_q$ is an eigenvalue of
multiplicity exactly two for the matrix with $t=t(\rho_q)$.
The coefficient $t(r)$ in \eqref{eq:negative-t} was chosen so that
$F_q(-r,t(r))=0$ for $0\le r\le1$.
For $0<r<1$, differentiating its formula gives
\begin{equation}
 t'(r)=\frac{r^{q-3}}{(1+r^{q-2})^2}
 \bigl((2q-1)r^{q+1}+(q+1)r^{2q-1}-(q-2)\bigr).
 \label{eq:terminal-derivative}
\end{equation}
The prefactor is positive. The expression in parentheses is strictly
increasing and, by \eqref{eq:odd-limit}, vanishes exactly at $r=\rho_q$.
Thus $t'(r)<0$ for $r<\rho_q$ and $t'(r)>0$ for $r>\rho_q$:
$t$ has a unique minimum at $\rho_q$.
Moreover, the derivative of the expression in parentheses is positive
there, so $t''(\rho_q)>0$.

Differentiating $F_q(-r,t(r))=0$ by the chain rule gives
\[
 -\frac{\partial F_q}{\partial X}(-r,t(r))
 +\frac{\partial F_q}{\partial t}(-r,t(r))\,t'(r)=0.
\]
At $r=\rho_q$, we have $t'(\rho_q)=0$, and hence
\[
 \frac{\partial F_q}{\partial X}(-\rho_q,t(\rho_q))=0.
\]
Together with $F_q(-\rho_q,t(\rho_q))=0$, this shows that
$-\rho_q$ is a root of multiplicity at least two.
Differentiating once more and again using $t'(\rho_q)=0$ gives
\[
 \frac{\partial^2 F_q}{\partial X^2}(-\rho_q,t(\rho_q))
 =-\frac{\partial F_q}{\partial t}(-\rho_q,t(\rho_q))\,
   t''(\rho_q)>0,
\]
because $\partial F_q/\partial t$ at $X=-r$ equals
$-r-r^{q-1}<0$.
The second derivative with respect to $X$ is therefore nonzero,
so the root has multiplicity exactly two.
Consequently, $-\rho_q$ is a double root of the characteristic polynomial
of $(1-t(\rho_q))U_{2q-1}+t(\rho_q)U_{q+1}$.
Thus this matrix realizes the endpoint reached by the curved boundary
as $x\uparrow1/2$.

For $q=5$, the equation \eqref{eq:odd-limit} defining the endpoint
radius is $9\rho^6+6\rho^9=3$. Its solution is
$\rho_5=2^{-1/3}$, and \eqref{eq:negative-t} gives $t(\rho_5)=3/4$.
Thus a matrix with the endpoint $-2^{-1/3}$ as an eigenvalue is
\begin{align}
 B_*&=B_{10}\bigl(0,0,\tfrac14e_1+\tfrac34e_4\bigr)
      =\tfrac14U_9+\tfrac34U_6,\label{eq:q5-matrix}\\
 \det(XI-B_*)&=X^{10}-\tfrac14X-\tfrac34X^4
      =X(X^3-1)\bigl(X^3+\tfrac12\bigr)^2.\label{eq:q5-factor}
\end{align}
At $X=-2^{-1/3}$, the factor $X^3+\tfrac12$ vanishes and its derivative
$3X^2$ is nonzero, while the factors $X$ and $X^3-1$ are nonzero.
Its square therefore makes $-2^{-1/3}$ a root of multiplicity exactly
two, and hence an eigenvalue of $B_*$ of algebraic multiplicity two.
Varying the weights in an average of $U_9$ and $U_6$ realizes every
point of $[-1,-2^{-1/3}]$, which lies on the boundary, and every point
of $(-2^{-1/3},0)$, which lies in the interior of $\Sigma_5$.

\subsection{The number of vertices needed to realize eigenvalues}\label{sec:sharpness}
For $d=0,1,2$, let $\Sq^{(d)}$ be the union of all faces of $\Sq$ of
dimension at most $d$.
Equivalently, these matrices have representations
\[
 A=\sum_{L=q}^{2q}w_LU_L,\qquad
 w_L\ge0,\quad\sum_{L=q}^{2q}w_L=1,
\]
with at most $d+1$ coefficients $w_L$ positive:
\begin{itemize}
\item $\Sq^{(0)}$ consists of the vertices themselves.
\item $\Sq^{(1)}$ consists of matrices $(1-t)U_M+tU_N$ with
$q\le M<N\le2q$ and $0\le t\le1$. It contains all edges and their
vertices; $t=0$ or $t=1$ gives an individual vertex.
\item $\Sq^{(2)}$ consists of convex combinations of three of the fixed
matrices, allowing coefficients to be zero. It contains all triangular
faces, together with their edges and vertices.
\end{itemize}
\begin{corollary}\label{cor:sharpness}
For every $q\ge2$, the matrices in $\Sq^{(1)}$ have the form
\[
 (1-t)U_M+tU_N,\qquad q\le M<N\le2q,\quad 0\le t\le1.
\]
These are convex combinations of two of the fixed matrices
$U_q,\ldots,U_{2q}$; $t=0$ or $t=1$ gives an individual matrix.
Their eigenvalues form a set with empty interior in the complex plane, and
\[
 \bigcup_{A\in\Sq^{(1)}}\spec A\subsetneq\Sigma_q.
\]
Thus some attainable eigenvalues cannot be realized by a convex
combination of two of the fixed matrices, although convex combinations
of three always suffice by Theorem~\ref{thm:triangles}.
Every boundary eigenvalue is an eigenvalue of a matrix of the displayed
form, but the individual matrices $U_q,\ldots,U_{2q}$ do not realize all
boundary eigenvalues.
\end{corollary}
\begin{proof}
Fix two distinct matrices $A_0,A_1$ from $U_q,\ldots,U_{2q}$ and
write their convex combinations as
\[
 A(t)=(1-t)A_0+tA_1,\qquad 0\le t\le1.
\]
Their characteristic polynomial $F(X,t)=\det(XI-A(t))$ has degree
$2q$ in $X$, with leading coefficient $1$, and its coefficients are
polynomials in $t$.
To track its distinct roots even when some remain repeated as $t$ varies,
take its square-free part $G(X,t)$ over $\C(t)$, normalized to have
leading coefficient $1$.
Here $\C(t)$ consists of rational functions of $t$, and taking the
square-free part means retaining one copy of each distinct polynomial
factor. Thus $G$ has the same distinct roots as $F$ except at finitely
many exceptional parameter values.

The discriminant of $G$ detects repeated roots: wherever its coefficients
are defined, the discriminant is zero exactly when $G$ has a repeated
root. It is a nonzero rational function of $t$ because $G$ is square-free.
Exclude its zeros and poles and the poles of the coefficients of $G$;
there are only finitely many such parameter values.
At every other value, each root has $\partial G/\partial X\ne0$, so the
implicit-function theorem gives it locally as an analytic function of $t$.

For real $t$, these functions trace curves in the complex plane.
On compact subintervals their derivatives are bounded, so their images
have planar area zero. A countable collection of these local curves
covers all roots at nonexceptional parameters.
The finitely many exceptional values contribute only finitely many
eigenvalues of $A(t)$.
Hence all eigenvalues obtained from this pair of matrices form a set
of planar area zero. There are only finitely many pairs, so their
combined eigenvalue set also has area zero and therefore contains no
open disk.

Theorem~\ref{thm:main} gives $\Sigma_q$ nonempty interior for every
$q\ge2$, so convex combinations of two of the fixed matrices cannot
realize all of $\Sigma_q$. Convex combinations of three do realize it by
Theorem~\ref{thm:triangles}.
Finally, the individual matrices $U_q,\ldots,U_{2q}$ have only finitely
many eigenvalues in total,
whereas $[0,1]\subseteq\partial\Sigma_q$.
Thus these individual matrices cannot realize the whole boundary.
Theorem~\ref{thm:main} gives a convex combination of two of them with
any prescribed boundary point as an eigenvalue.
\end{proof}
Corollary~\ref{cor:sharpness} concerns the number of fixed matrices
$U_q,\ldots,U_{2q}$ that must appear with positive coefficients in a
convex combination. It does not minimize the number of nonzero entries
of a realizing matrix among all stochastic matrices.
For a boundary eigenvalue, the corollary guarantees a realizing matrix
of the form $(1-t)U_M+tU_N$; it does not assert that every matrix
with that eigenvalue has this form.

\subsection{Computing boundary radii and constructing realizing matrices}\label{sec:algorithm}
Computing the boundary radius means solving for $\rho$ at a fixed
angle. The construction below for a prescribed nonreal point $\lambda$
with $q\ge3$ instead searches for a real parameter $s$ for which
$\Gamma_\lambda(s)\in K_q(\lambda)$, as in \eqref{eq:transfer}.

To compute the boundary radius, fix $x$ strictly between consecutive
fractions in $\F_q$. Set $(u,v)=(M,N)$ if $M,N>q$, and
$(u,v)=(q,L/2)$ if one endpoint length is $q$ and the other is $L$.
Use the angles $\alpha,\beta$ in \eqref{eq:base-data} or
\eqref{eq:lateral-data}, respectively, and define
\[
 f(\rho)=\rho^v\sin\alpha+\rho^u\sin\beta-\sin(\alpha+\beta).
\]
By Lemma~\ref{lem:chord}, $f$ is strictly increasing on $[0,1]$ and
$f(0)<0<f(1)$. Bisection retains the half-interval on whose endpoints
$f$ has opposite signs, and converges to its unique zero $r_q(x)$.
At this radius, use $t$ from \eqref{eq:base-t} to construct
$(1-t)U_M+tU_N$, or use $c$ from \eqref{eq:lateral-equation}, with the
root in \eqref{eq:root}, to construct $cU_q+(1-c)U_L$.
For unit-circle points, real points, and points $re^{\pm i\pi/q}$ with
$0\le r\le1$, use the explicit matrices in the preceding subsection.

Let $q\ge3$ and let $\lambda$ be nonreal.
Describe $K_q(\lambda)$ by one inequality for each of its $h$ edges:
\[
 K_q(\lambda)=\{w\in\C:\Re(\overline{n_\ell}w)\le b_\ell,
 \quad1\le\ell\le h\},\qquad |n_\ell|=1.
\]
Here $n_\ell$ represents a unit vector perpendicular to edge $\ell$
and pointing away from the polygon. The line containing that edge is
$\Re(\overline{n_\ell}w)=b_\ell$; the inequality keeps the side
containing the polygon.

Put $d=\lambda^q-1$. By \eqref{eq:transfer},
\[
 \Gamma_\lambda(s)=(1+sd)^2=1+2sd+s^2d^2.
\]
Substituting this expression for $w$ in each edge inequality shows that
$\Gamma_\lambda(s)\in K_q(\lambda)$ exactly when all the following
inequalities hold:
\begin{equation}
 A_\ell s^2+B_\ell s+C_\ell\le0,
 \quad\begin{cases}
 A_\ell=\Re(\overline{n_\ell}d^2),\\
 B_\ell=2\Re(\overline{n_\ell}d),\\
 C_\ell=\Re(\overline{n_\ell})-b_\ell.
 \end{cases}\label{eq:quadratic-feasibility}
\end{equation}
When $|\lambda|\le1$ and $\lambda^q\ne1$, the criterion
\eqref{eq:criterion} therefore reduces attainability to finding an
$s\in[1,s^*(\lambda)]$ satisfying every inequality in
\eqref{eq:quadratic-feasibility}. The bound in \eqref{eq:sstar} is
\[
 s^*(\lambda)=\frac{2(1-\Re\lambda^q)}{|1-\lambda^q|^2}.
\]
The polynomials have degree at most two, so their signs can be checked
on finitely many intervals and at their endpoints.
The procedure below also checks linear and constant inequalities and
retains isolated acceptable values of $s$.
It does not assume that all acceptable values form a single interval.

\begin{center}
\setlength{\fboxsep}{9pt}
\fbox{\begin{minipage}{0.91\linewidth}
\textbf{Testing attainability and constructing a realizing matrix.}
For $q\ge3$ and $\lambda\in\D\setminus\R$ with $\lambda^q\ne1$:
\begin{enumerate}[leftmargin=*,itemsep=4pt,topsep=5pt]
\item Compute $K_q(\lambda)=\conv\{1,\lambda,\ldots,\lambda^{q-1}\}$.
For each edge, write the inequality
$\Re(\overline{n_\ell}w)\le b_\ell$ described above.
Substitution of $\Gamma_\lambda(s)$ gives the polynomial inequalities
\eqref{eq:quadratic-feasibility}.
\item Find the real zeros of the nonconstant polynomials in
$[1,s^*(\lambda)]$ and add the two interval endpoints.
Sort the distinct values. Test every value and one point in each open
interval between consecutive values against \emph{all} inequalities.
The polynomials keep their signs between consecutive zeros, so one
test determines whether that whole open interval is acceptable.
Check constant inequalities separately: a positive constant rules out
$\lambda$, while a nonpositive constant imposes no restriction on $s$.
Retain individual acceptable values even when the adjacent intervals fail.
If no value is acceptable, then $\lambda\notin\Sigma_q$.
Otherwise, choose the largest acceptable value $s_{\max}$; it exists
because the acceptable set is closed in $[1,s^*(\lambda)]$.
\item Set $w_0=\Gamma_\lambda(s_{\max})$.
As shown in the proof of Theorem~\ref{thm:triangles}, this point lies on
the boundary of $K_q(\lambda)$.
Choose an edge $[\lambda^i,\lambda^j]$ containing $w_0$ and compute
\[
 \eta=\frac{\Re\bigl((w_0-\lambda^i)
             \overline{(\lambda^j-\lambda^i)}\bigr)}
              {|\lambda^j-\lambda^i|^2}\in[0,1].
\]
Then $w_0=(1-\eta)\lambda^i+\eta\lambda^j$.
Set $p=(1-\eta)e_i+\eta e_j$ and $c=1-1/s_{\max}$.
The required matrix is
\[
 B_{2q}(c,c,p)
 =cU_q+(1-c)\bigl((1-\eta)U_{2q-i}+\eta U_{2q-j}\bigr).
\]
It has $\lambda$ as an eigenvalue, with the eigenvector in
\eqref{eq:eigenvector} evaluated at $s=s_{\max}$.
\end{enumerate}
\end{minipage}}
\end{center}
If $w_0$ is a vertex of $K_q(\lambda)$, either of the two edges meeting
there may be chosen in Step~3. The coefficient $\eta$ is then $0$ or $1$,
so one of the two coefficients associated with that edge vanishes.
The procedure checks both intervals and isolated acceptable values of
$s$; it does not require those values to form a single interval.

If $\lambda^q=1$, choose $U_q$.
For real $\lambda\in[-1,1]$, choose
$c=(1+\lambda^q)/2$ and use $cU_q+(1-c)U_{2q}$, as in
\eqref{eq:first-edge}.
Appendix~\ref{app:four} gives direct parameters for nonreal eigenvalues
when $q=2$.
Finally, every realizing matrix has real entries, so the same matrix
that realizes $\lambda$ also realizes $\overline\lambda$.

\section{Comparison with all $2q\times2q$ stochastic matrices}\label{sec:ambient}
Recall from Section~\ref{sec:introduction} that $\Theta_{2q}$ is the
set of eigenvalues of all $2q\times2q$ row-stochastic matrices, with no
restriction to the transition rules of our family.
Ito's polynomial equations describe its boundary curves.
After common powers of the variable are removed, the resulting
polynomials are called reduced Ito polynomials.
We compare them with \eqref{eq:base-equation} and
\eqref{eq:lateral-equation}, keeping three questions separate:
\begin{enumerate}
\item Do the polynomials have the same algebraic form, allowing the
integer exponents to differ?
\item For the same pair of consecutive Farey fractions, do the integer
exponents agree with those in the equation for the boundary of
$\Theta_{2q}$ between those endpoint angles?
\item Do the roots describing our boundary trace exactly the same curve
between those angles as the boundary of $\Theta_{2q}$?
\end{enumerate}
The first question is answered by comparing polynomial expressions.
Identifying the boundary curves also requires realizing their points
within our family and comparing the largest attainable radius at each
angle. The classification of $\Sigma_q$ in Theorem~\ref{thm:main} was
proved independently of this comparison.

\subsection{Polynomial forms and the substitution $Y=X^g$}
Consider the convex combination $\beta U_m+\alpha U_n$, where
$q<m<n\le2q$, $\alpha,\beta\ge0$ and $\alpha+\beta=1$.
Its characteristic polynomial is
\begin{equation}
 X^{2q-n}\bigl(X^n-\beta X^{n-m}-\alpha\bigr).
 \label{eq:ito-base}
\end{equation}
Thus its nonzero eigenvalues satisfy the three-term equation
$X^n-\beta X^{n-m}-\alpha=0$.
Because $m>n/2$, this is the polynomial form called Type~I in the
classification of reduced Ito polynomials \cite{ito,johnson,kirkland}.
The two cycle lengths $m,n$ need not, however, be relatively prime.
Let $g=\gcd(m,n)$, their greatest common divisor. With $Y=X^g$, write
\[
 X^n-\beta X^{n-m}-\alpha=F(X^g),\qquad
 F(Y)=Y^{n/g}-\beta Y^{(n-m)/g}-\alpha.
\]
If $g>1$, this substitution reduces the degree from $n$ to $n/g$.
The reduced cycle lengths $m/g,n/g$ have no common divisor greater than
one, as required for the denominators of two consecutive Farey fractions.
This explains the algebraic reduction; identifying an actual boundary
curve requires the additional comparisons below.

For the convex combination $cU_q+(1-c)U_L$, where $q<L<2q$ and
$0\le c\le1$, the characteristic polynomial is
\begin{equation}
 (X^q-c)^2-(1-c)^2X^{2q-L}.\label{eq:ito-lateral}
\end{equation}
This has the Type~II reduced-Ito form: a power of a binomial minus a
multiple of a single power of $X$. Here the exponent on the binomial is
$2$; it does not mean that every eigenvalue is a double root.
Let $g=\gcd(q,L)=\gcd(q,2q-L)$. Substituting $Y=X^g$ gives
\[
 (Y^{q/g}-c)^2-(1-c)^2Y^{(2q-L)/g},
\]
with reduced cycle lengths $q/g,L/g$ having no common divisor greater
than one. Joshi, Kirkland and \v{S}migoc study when raising points on
one Karpelevi\v{c} boundary curve to an integer power produces another
such curve, and when the corresponding realizing matrices are themselves
powers of stochastic matrices \cite{powers}.
Our substitutions above establish polynomial identities, but do not
alone identify which roots trace the boundary of $\Theta_{2q}$.
For our boundary realized by $cU_q+(1-c)U_L$, we must also retain the
square root specified in \eqref{eq:root}. Explicitly, for
$\lambda=\rho e^{2\pi ix}$ and the Farey endpoint $\ell/L$ associated
with $U_L$, we use
\[
 \xi=(-1)^\ell\rho^{(2q-L)/2}e^{\pi i(2q-L)x},\qquad
 \lambda^q=c+(1-c)\xi.
\]
The second identity is \eqref{eq:lateral-equation}; squaring it after
subtracting $c$ gives \eqref{eq:ito-lateral}, but loses the choice
between the two square roots of $\lambda^{2q-L}$.

\subsection{When the two regions have the same boundary curve}
\begin{proposition}\label{prop:ambient-agreement}
Let $f<g$ be consecutive fractions in $\F_q$, with denominators $k,s$.
If $k,s\ge q$, the curve
\[
 \{r_q(x)e^{2\pi ix}:f<x<g\}
\]
is exactly the portion of $\partial\Theta_{2q}$ with arguments
$2\pi x$, $f<x<g$.
If both denominators exceed $q$, the common equation for nonzero
eigenvalues is the three-term Type~I equation in \eqref{eq:ito-base}.
If the smaller denominator equals $q$, the common equation is the
Type~II equation with a squared binomial in \eqref{eq:ito-lateral}.
\end{proposition}
\begin{proof}
For this proof, relabel the two denominators so that $k<s$, regardless
of which fraction comes first. Since $k,s\ge q$, their associated cycle
lengths in \eqref{eq:farey} are already $k,s$:
$\ell_q(k)=k$ and $\ell_q(s)=s$.
Ito's equation for the boundary of $\Theta_{2q}$ between the two endpoint
angles is
\begin{equation}
 X^s(X^k-\beta)^d=\alpha^dX^{kd},\qquad
 d=\lfloor 2q/k\rfloor,\quad \alpha+\beta=1,
 \label{eq:ambient-ito}
\end{equation}
where $0\le\alpha\le1$ \cite{ito,johnson}.
If $k>q$, then $d=1$. Cancelling the common factor $X^k$ for $X\ne0$
gives
\[
 X^s-\beta X^{s-k}-\alpha=0.
\]
By \eqref{eq:ito-base}, this is the equation for the nonzero eigenvalues
of $\beta U_k+\alpha U_s$.

If $k=q$, then $d=2$. Also $q<s<2q$: the equality $s=2q$ is excluded
because the denominators of consecutive Farey fractions have no common
divisor greater than one. Cancelling $X^s$ for $X\ne0$ gives
\[
 (X^q-\beta)^2-\alpha^2X^{2q-s}=0.
\]
This is \eqref{eq:ito-lateral} with $L=s$ and $c=\beta$, the
characteristic equation of $\beta U_q+\alpha U_s$.

In both cases, $0\le\alpha\le1$ and $\beta=1-\alpha$ give convex
combinations of our fixed matrices throughout the parameter range in
Ito's equation. Every nonzero point on this boundary curve of
$\Theta_{2q}$ is therefore an eigenvalue of a matrix in our family,
and so belongs to $\Sigma_q$.
Conversely, every matrix in our family is a $2q\times2q$ row-stochastic
matrix, hence $\Sigma_q\subseteq\Theta_{2q}$.
Fix $x\in(f,g)$ and let $R$ be the largest modulus attained in
$\Theta_{2q}$ on the ray with argument $2\pi x$.
The point $Re^{2\pi ix}$ lies on the boundary curve just realized in
$\Sigma_q$, while the inclusion prevents any point of $\Sigma_q$ on
this ray from having modulus greater than $R$.
Thus the largest attainable modulus in $\Sigma_q$ is also $R$.
Theorem~\ref{thm:main} identifies this modulus as $r_q(x)$, proving the
claimed equality of boundary curves.
\end{proof}
If the smaller Farey denominator is $k<q$, the associated cycle length
in our family is instead $\ell_q(k)=k\lceil q/k\rceil>k$, by
\eqref{eq:farey}. The two cycle lengths, even after division by their
greatest common divisor, need not be the original Farey denominators.
Thus a three-term or squared-binomial polynomial can have the same
algebraic form as an Ito polynomial without describing the boundary of
$\Theta_{2q}$ between the same endpoint angles. Identifying our boundary
still requires the classification in Section~\ref{sec:contacts}.

For odd $q\ge3$, this distinction is visible on the final interval
\[
 \frac{q-1}{2q-1}<x<\frac12.
\]
The smaller denominator is $2$, and its associated cycle length is
$\ell_q(2)=q+1$; the other cycle length is $2q-1$.
Proposition~\ref{prop:limits} gives
\[
 \lim_{x\uparrow1/2}r_q(x)e^{2\pi ix}=-\rho_q,
 \qquad 0<\rho_q<1,
\]
where $\rho_q$ is defined by \eqref{eq:odd-limit}.
The boundary curve of $\Theta_{2q}$ between these same endpoint angles
ends at $e^{2\pi i(1/2)}=-1$.

The two regions contain exactly the same points on the unit circle:
the roots of unity whose orders are at most $2q$, by \eqref{eq:unit}
and Karpelevi\v{c}'s theorem \cite{ito}. The regions nevertheless differ.
To exhibit a point in $\Theta_{2q}\setminus\Sigma_q$, let $C_{2q}$ be
the permutation matrix for a single cycle through all $2q$ states, and
let $I$ be the identity matrix of the same order. For $0<\tau<1$, set
\[
 A_\tau=(1-\tau)I+\tau C_{2q}.
\]
This is a row-stochastic matrix: at each state, the chain stays in place
with probability $1-\tau$ and moves to the next state in the cycle with
probability $\tau$. Since $e^{i\pi/q}$ is an eigenvalue of $C_{2q}$,
the same eigenvector gives the eigenvalue
\[
 z_\tau=(1-\tau)+\tau e^{i\pi/q}
\]
of $A_\tau$. This point lies strictly between $1$ and $e^{i\pi/q}$ on
the straight line segment joining them, so $0<\Arg z_\tau<\pi/q$.
Proposition~\ref{prop:elementary} excludes every such point from
$\Sigma_q$, whereas $A_\tau$ shows that it belongs to $\Theta_{2q}$.
Therefore $\Sigma_q\subsetneq\Theta_{2q}$.
The agreement of boundary curves in Proposition~\ref{prop:ambient-agreement}
is thus only partial. Theorem~\ref{thm:main} describes the full region
$\Sigma_q$ and realizing matrices for the specific transition rules in
\eqref{eq:model}.

\section{Concluding discussion}\label{sec:conclusion}
The region $\Sigma_q$ shows how prescribed transitions can change the
possible eigenvalues while retaining all unit-circle eigenvalues of
unrestricted stochastic matrices of the same order. Some unrestricted
boundary arcs also remain attainable, as proved in
Proposition~\ref{prop:ambient-agreement}. The excluded sector and the
odd-$q$ endpoint inside the unit disk show where the full regions differ.
Theorem~\ref{thm:main} determines every attainable radius and supplies
realizing matrices. Although the balanced family has dimension $q$,
three of its fixed matrices suffice for any specified eigenvalue, and
two for a boundary point. These are sharp bounds for this fixed
representation, rather than bounds on the number of nonzero entries of
an arbitrary stochastic matrix.

The balancing argument and the boundary geometry have different scopes.
As explained after Proposition~\ref{prop:balancing}, the same algebraic
replacement of $a,b$ by $c=(a+b)/2$ works for any permitted class of
polynomials in the characteristic formula that contains $1$ and is
closed under convex combinations. For $c<1$, balancing replaces $P$ by
$\tau P+(1-\tau)$, with $\tau\in[0,1]$ defined in \eqref{eq:balance};
if $c=1$, the probabilities are already equal. The later boundary analysis uses
the particular hull $\conv\{1,\lambda,\ldots,\lambda^{q-1}\}$ and the
cycle lengths imposed by \eqref{eq:model}. Extending the classification
to sequences with different numbers of states, or allowing both last
states to choose arbitrary destinations, would therefore require new
arguments.

\appendix
\section{Proof of the lattice lemma}\label{app:lattice}
We prove Lemma~\ref{lem:lattice} by grouping the integer points in
$H(P,Q)$ according to the value of $\det((r,p),U)$, then looking for a
point $U=(m,n)\ne Q$ with $q<m\le2q$ in one of these groups.
The hypothesis $Q\notin2\Z^2$ in \eqref{eq:non-even} means that $L,b$
are not both even. Its role in the geometric application is to ensure
that $P+Q/2$ is not an integer point, excluding one of the possible
equality cases in the proof of Proposition~\ref{prop:lateral}.
\begin{proof}[Proof of Lemma~\ref{lem:lattice}]
Put
\[
 d=\gcd(q,a),\qquad P=dP_*,\qquad
 P_*=(r,p)=(q/d,a/d),\qquad \delta=\det(P_*,Q)>0.
\]
The coordinates of $P_*$ have greatest common divisor $1$, so we can
choose an integer vector $C$ with $\det(P_*,C)=1$.
Consequently, every integer point has a unique expression
$U=xP_*+yC$ with $x,y\in\Z$.
In these coordinates, write $Q=\kappa P_*+\delta C$, where
$\kappa\in\Z$. The coordinate $y$ is exactly $\det(P_*,U)$.

For $U=\sigma P+\tau Q$, we have
$x=d\sigma+\kappa\tau$ and $y=\delta\tau$.
Thus $0\le\sigma\le1$ and $1/2\le\tau\le1$ are equivalent to
\begin{equation}
 \lceil\delta/2\rceil\le y\le\delta,\qquad
 0\le\delta x-\kappa y\le d\delta.\label{eq:lattice-membership}
\end{equation}
Fix an integer $y$ satisfying $\lceil\delta/2\rceil\le y\le\delta$.
The smallest integer $x$ satisfying the second inequality in
\eqref{eq:lattice-membership} is
\begin{equation}
 x_y=\lceil\kappa y/\delta\rceil,\qquad
 \nu_y=\delta x_y-\kappa y\in\{0,\ldots,\delta-1\},\qquad
 U_y=x_yP_*+yC.\label{eq:lattice-level}
\end{equation}
Since $U_y=(\nu_y/\delta)P_*+(y/\delta)Q$, its first coordinate
in the original plane, denoted $m_y$, is
\begin{equation}
 m_y=\frac{yL+\nu_yr}{\delta}.\label{eq:lattice-first}
\end{equation}
All integer points in $H$ with this coordinate $y$ are obtained by
adding $kP_*$ to $U_y$. Equation~\eqref{eq:lattice-membership} allows
$0\le k\le d-1$, and also $k=d$ precisely when $\nu_y=0$.
Each addition of $P_*$ increases the first coordinate by $r$.

If $\delta=1$, the only possible integer $y$ is $1$, and the integer
points in $H$ are
\[
 Q,Q+P_*,\ldots,Q+dP_*.
\]
Their first coordinates are $L,L+r,\ldots,L+dr$.
Since $q<L\le2q$, a point other than $Q$ has first coordinate in
$(q,2q]$ if and only if $L+r\le2q$.
Thus $Q$ is the only such point exactly when $L+r>2q$, as claimed
in the second alternative of the lemma.
For the rest of the proof, assume $\delta\ge2$; we will find a point
other than $Q$ with the required first coordinate.

First suppose $d\ge2$ and choose $y_0=\lceil\delta/2\rceil<\delta$.
Using $m_y=(yL+\nu_yr)/\delta$ from \eqref{eq:lattice-first} gives
\[
 m_{y_0}\ge L/2>q/2\ge r,\qquad
 m_{y_0}<\frac{y_0L}{\delta}+r
 \le q+\frac q\delta+\frac qd\le2q.
\]
If $m_{y_0}>q$, the point $U_{y_0}$ works.
Otherwise choose the smallest integer $k\ge0$ such that
$m_{y_0}+kr>q$. Since $q=dr$ and $m_{y_0}>r$, we have $k\le d-1$,
so this translate belongs to $H$. Minimality gives
\[
 q<m_{y_0}+kr\le q+r\le2q.
\]
Its coordinate $y_0<\delta$ distinguishes it from $Q$, whose coordinate
is $\delta$.

It remains to treat $d=1$, so $P=P_*$ and $r=q$.
Set $T=q\delta/L\in[\delta/2,\delta)$.
In \eqref{eq:lattice-first}, the term $yL/\delta$ equals $q$ when
$y=T$; we therefore examine the integer values at or immediately
beside $T$.
If $T$ is an integer, take $y=T$. Then
\[
 m_y=q+\frac{\nu_yq}{\delta}.
\]
If $\nu_y>0$, this lies in $(q,2q)$.
If $\nu_y=0$, the allowed point $U_y+P$ has first coordinate $2q$.
Both points differ from $Q$ because $y<\delta$.

Now suppose $T$ is not an integer, and put
$y_-=\lfloor T\rfloor$, $y_+=\lceil T\rceil$.
If $y_+=\delta$, then
$\lceil\delta/2\rceil\le y_-=\delta-1<\delta$, so $y_-$ satisfies
both bounds on $y$ in \eqref{eq:lattice-membership}.
Since $y_-<T$ and $\nu_{y_-}<\delta$, equation
\eqref{eq:lattice-first} gives $m_{y_-}<2q$. If $m_{y_-}>q$, take $U_{y_-}$.
Otherwise \eqref{eq:lattice-first} gives
\[
 \nu_{y_-}q\le q\delta-(\delta-1)L<q,
\]
so the nonnegative integer $\nu_{y_-}$ must be zero.
The allowed translate $U_{y_-}+P$ then has first coordinate in
$(q,2q]$, and its coordinate $y_-<\delta$ shows that it differs from $Q$.

If $y_-<\lceil\delta/2\rceil$, then $\delta$ is odd and
$y_+=(\delta+1)/2=\lceil\delta/2\rceil<\delta$.
Thus $y_+$ satisfies both bounds on $y$ in
\eqref{eq:lattice-membership}.
Since $y_+>T$, equation \eqref{eq:lattice-first} gives
\[
 q<m_{y_+}\le
 \frac{(\delta+1)q+(\delta-1)q}{\delta}=2q.
\]
Thus $U_{y_+}$ works.

In the remaining case, both integers satisfy
$\lceil\delta/2\rceil\le y_-<y_+<\delta$.
Equation~\eqref{eq:lattice-first}, with $y_-<T<y_+$, gives
$m_{y_-}<2q$ and $m_{y_+}>q$.
If neither point has first coordinate in $(q,2q]$, then
$m_{y_-}\le q$ and $m_{y_+}>2q$.
Put $V=U_{y_+}-U_{y_-}$, let $h$ be its first coordinate, and set
$\sigma=\nu_{y_+}-\nu_{y_-}$. These inequalities give $h>q$.
Because $y_+-y_-=1$, the coordinate expressions yield
\[
 \det(P,V)=1,\qquad Q=-\sigma P+\delta V.
\]
Taking first coordinates gives $L=-\sigma q+\delta h$.
Together with $h>q$ and $L\le2q$, this implies $\delta<\sigma+2$.
But $\sigma\le\delta-1$, so $\sigma=\delta-1$.
Since both $\nu$ values belong to $\{0,\ldots,\delta-1\}$, this forces
$\nu_{y_-}=0$ and $\nu_{y_+}=\delta-1$.

Since $\det(P,V)=1$, every integer point also has integer coefficients
in the basis $(P,V)$. Write
$U_{y_-}=uP+y_-V$ with $u\in\Z$.
By \eqref{eq:lattice-level},
$\det(U_{y_-},Q)=\nu_{y_-}=0$. Substituting the expressions in this
basis gives
\[
 u\delta+y_-(\delta-1)=0.
\]
Reducing this equality modulo $\delta$ shows that $\delta$ divides
$y_-$, contradicting $0<y_-<\delta$.
Thus this last possibility cannot occur, and a required point exists
whenever $\delta\ge2$.
For negative determinant, negate the second coordinate of every point.
This changes the determinant's sign and preserves the first coordinates,
so the same conclusion applies.
\end{proof}

\section{The four-state case}\label{app:four}
For $q=2$, the convex hull in \eqref{eq:transfer} is just the segment
$K_2(\lambda)=[1,\lambda]$. This segment has empty interior in the
complex plane, so the argument in Lemma~\ref{lem:persistence} based on
an interior point of $K_q(\lambda)$ cannot be used.
We instead solve the geometric criterion explicitly in real and imaginary
coordinates. The real points are already $[-1,1]$ by
Proposition~\ref{prop:elementary}, and conjugation determines the lower
half-plane from the upper half-plane.

Write $\lambda=X+iy$ with $y>0$, and set $Y=y^2>0$.
By Proposition~\ref{prop:transfer}, $\lambda$ is attainable if and only
if it lies in the unit disk and there are $s\ge1$ and $0\le t\le1$ with
\[
 \Gamma_\lambda(s)=\bigl(1+s(\lambda^2-1)\bigr)^2
 =1+t(\lambda-1).
\]
Here $t$ is the coefficient of $\lambda$ in the convex combination
$(1-t)\cdot1+t\lambda$; it corresponds to the probability vector
$p=(1-t,t)$. Since $\lambda\ne1$, the required coefficient is
\[
 \frac{\Gamma_\lambda(s)-1}{\lambda-1}
 =s(\lambda+1)\bigl(2+s(\lambda^2-1)\bigr).
\]
Put $D=1-2X-3X^2+Y$. The imaginary part of this expression is
$sy(2-sD)$. Since $s\ge1$ and $y>0$, a real coefficient requires
$2-sD=0$; hence $D>0$ and $s=2/D$.
Substituting this value into the real part gives
\begin{equation}
 s=\frac2{1-2X-3X^2+Y},\qquad
 t=\frac{-8X\bigl((X+1)^2+Y\bigr)}{(1-2X-3X^2+Y)^2}.\label{eq:four-parameters}
\end{equation}
Thus, for a fixed nonreal $\lambda$, these parameters are determined
uniquely if a solution exists.

The unit-disk condition is $X^2+Y\le1$. Since $Y>0$, it gives
$-1<X<1$. In \eqref{eq:four-parameters}, the factor
$(X+1)^2+Y$ is positive, so $t\ge0$ is equivalent to $X\le0$.
For $-1<X\le0$, the denominator
$D=(1+X)(1-3X)+Y$ is positive.
With this positivity established, $s\ge1$ is equivalent to
$D\le2$, or $Y\le1+2X+3X^2$.
Finally, $t\le1$ is equivalent to
$G_X(Y):=D^2+8X\bigl((X+1)^2+Y\bigr)\ge0$.
Consequently, for $Y>0$, attainability is equivalent to the four inequalities
\begin{equation}
 X\le0,\qquad X^2+Y\le1,\qquad
 Y\le1+2X+3X^2,\qquad G_X(Y)\ge0.\label{eq:four-conditions}
\end{equation}
Expanding the expression defining $G_X$ gives
\begin{equation}
 G_X(Y)=Y^2+(-6X^2+4X+2)Y+9X^4+20X^3+14X^2+4X+1.\label{eq:four-poly}
\end{equation}
There is no further restriction from the upper bound $s\le s^*(\lambda)$
in \eqref{eq:sstar}: the equality
$\Gamma_\lambda(s)\in[1,\lambda]\subseteq\D$ already implies it.
Whenever \eqref{eq:four-conditions} holds, the values in
\eqref{eq:four-parameters} therefore give a realizing matrix
$B_4(c,c,p)$ with $c=1-1/s$ and $p=(1-t,t)$.

We now solve the inequalities in \eqref{eq:four-conditions} for $Y>0$.
For $-1/2\le X\le0$, the discriminant of the monic quadratic $G_X(Y)$
with respect to $Y$ is $64X^2(-2X-1)\le0$.
Thus $G_X(Y)\ge0$ for every real $Y$.
Also $1+2X+3X^2\le1-X^2$ on this interval, so the bound
$Y\le1+2X+3X^2$ implies the unit-disk bound $Y\le1-X^2$.
The attainable points here are therefore exactly those with
$0<Y\le1+2X+3X^2$, and their upper boundary is
\begin{equation}
 Y=1+2X+3X^2,\qquad-\tfrac12\le X\le0.\label{eq:four-base}
\end{equation}
For $-1\le X\le-1/2$, the two roots of $G_X$ are
\[
 \begin{aligned}
 Y_-&=(X-1)(3X+1)+4X\sqrt{-2X-1},\\
 Y_+&=(X-1)(3X+1)-4X\sqrt{-2X-1}.
 \end{aligned}
\]
Since $G_X$ is monic, $G_X(Y)\ge0$ holds exactly when
$Y\le Y_-$ or $Y\ge Y_+$.
To compare these alternatives with the other bounds in
\eqref{eq:four-conditions}, put $w=\sqrt{-2X-1}\in[0,1]$.
The following factorizations give the required comparisons:
\[
 \begin{aligned}
 Y_-&=\tfrac14(1-w)^2(3w^2-2w+3),\\
 1-X^2-Y_-&=w(1-w)(w^2-w+2),\\
 1+2X+3X^2-Y_-&=2w(w^2-w+1),\\
 Y_+-(1-X^2)&=w(w+1)(w^2+w+2).
 \end{aligned}
\]
The first three expressions are nonnegative, so $Y_-$ is nonnegative
and no larger than either upper bound $1-X^2$ or $1+2X+3X^2$.
For $X<-1/2$, we have $w>0$, and the fourth expression is positive:
$Y_+>1-X^2$. Thus the alternative $Y\ge Y_+$ is excluded by the
unit-disk condition. At $X=-1/2$, the two roots coincide.
Consequently, \eqref{eq:four-conditions} reduces to
\begin{equation}
 0<Y\le Y_-(X),\qquad-1<X\le-\tfrac12.\label{eq:four-lateral}
\end{equation}
These inequalities, together with $0<Y\le1+2X+3X^2$ for
$-1/2\le X\le0$, describe every attainable eigenvalue in the upper
half-plane with positive imaginary part.
At $X=0$, they give $0<y\le1$; adjoining the already attainable real
point $0$ gives the segment $[0,i]$.

On the boundary curve \eqref{eq:four-base}, substituting
$Y=1+2X+3X^2$ into \eqref{eq:four-parameters} gives $s=1$ and hence
$c=1-1/s=0$. The geometric criterion therefore becomes
\[
 \lambda^4=(1-t)+t\lambda,
\]
and the realizing matrix is
\[
 B_4(0,0,(1-t,t))=(1-t)U_4+tU_3,\qquad0\le t\le1,
\]
with $t$ given by \eqref{eq:four-parameters}.
At $X=0$, the curve reaches $\lambda=i$ and $t=0$; at $X=-1/2$,
it reaches $\lambda=e^{2\pi i/3}$ and $t=1$.
For $-1/2<X<0$, we have $X<0$, $y>0$ and
$Y-3X^2=1+2X>0$. Thus $y>-\sqrt3X$, which gives
$\pi/2<\Arg\lambda<2\pi/3$.
In the normalized angle $x=\Arg\lambda/(2\pi)$, this is precisely
$1/4<x<1/3$, the interval between the consecutive Farey fractions
associated with $U_4$ and $U_3$.

On the second boundary curve, $Y=Y_-(X)$ in
\eqref{eq:four-lateral}, we have $G_X(Y)=0$ and hence $t=1$ by the
parameter calculation preceding \eqref{eq:four-conditions}.
The probability vector is therefore $p=(0,1)$, and the realizing matrix is
\[
 B_4(c,c,(0,1))=cU_2+(1-c)U_3,\qquad c=1-1/s.
\]
Its characteristic equation is
\[
 (\lambda^2-c)^2=(1-c)^2\lambda.
\]
For $-1<X<-1/2$, the bounds $Y\le1-X^2<3X^2$ give
$y<-\sqrt3X$. Since $X<0$ and $y>0$, this means
$2\pi/3<\Arg\lambda<\pi$.

To determine which square root occurs before squaring, define
\[
 \xi=1+s(\lambda^2-1).
\]
The equality $\Gamma_\lambda(s)=\lambda$ gives $\xi^2=\lambda$, and
$c=1-1/s$ gives $\lambda^2=c+(1-c)\xi$.
Moreover,
\[
 \Im\xi=2sXy<0.
\]
Write $\lambda=\rho e^{i\theta}$ with $\theta=\Arg\lambda\in(0,\pi)$.
Its two square roots are $\pm\rho^{1/2}e^{i\theta/2}$.
The root $\rho^{1/2}e^{i\theta/2}$ has positive imaginary part,
so the required root is
\[
 \xi=-\rho^{1/2}e^{i\theta/2}.
\]
This is exactly \eqref{eq:root} with $q=2$, $L=3$ and $\ell=1$.
Thus $cU_2+(1-c)U_3$ realizes the boundary between the Farey fractions
$1/3$ and $1/2$, using the specified square root in
\eqref{eq:lateral-equation}.

The Cartesian description in \eqref{eq:four-base} and
\eqref{eq:four-lateral} places every boundary point with $X<0$ and
$y>0$ on one of the two upper curves just identified.
For each angle in $(\pi/2,2\pi/3)$ or $(2\pi/3,\pi)$,
Lemma~\ref{lem:chord} and Proposition~\ref{prop:carrier}, whose
construction also applies to $q=2$, give exactly one radius on the
applicable curve. Denote it by $R=r_2(\theta/(2\pi))$.
We show that the attainable positive radii on this ray are exactly
$(0,R]$.

Fix such an angle $\theta$ and write
$X=\rho\cos\theta$, $Y=\rho^2\sin^2\theta$.
For every sufficiently small $\rho>0$, we have $X<0$ and
\[
 1-X^2-Y>0,\qquad 1+2X+3X^2-Y>0,\qquad G_X(Y)>0,
\]
because each of the three displayed expressions tends to $1$ as
$\rho\downarrow0$.
Thus \eqref{eq:four-conditions} makes every sufficiently small positive
radius attainable.

By compactness of $\Sigma_2$ from Proposition~\ref{prop:elementary},
the largest attainable radius exists. Its point is on the boundary,
because points farther out on the same ray are unattainable.
The uniqueness of the boundary radius therefore makes this maximum $R$.
If a radius between $0$ and $R$ were missing, closedness would give
an open interval $(a,b)$ of unattainable radii with
$0<a<b\le R$ and both endpoints attainable.
The two points $ae^{i\theta}$ and $be^{i\theta}$ would then be boundary
points: each is approached along the ray by unattainable points.
This contradicts the uniqueness of the boundary radius.
Hence every radius $0<\rho\le R$ is attainable, as claimed.

The two upper boundary curves meet at $e^{2\pi i/3}$ and end at $i$
and $-1$. Consider the straight line segment from the origin to their
common endpoint $e^{2\pi i/3}$. Its points are
$\lambda=\rho e^{2\pi i/3}$ with $0\le\rho\le1$:
$\rho$ is the distance from the origin, and each nonzero point has angle
$2\pi/3$.
The calculation preceding \eqref{eq:four-base} showed that
$\lambda=X+iy$ with $y>0$ is attainable whenever
\[
 -\tfrac12\le X\le0,\qquad 0<Y=y^2\le1+2X+3X^2.
\]
For our points with $0<\rho\le1$, we have
\[
 X=-\rho/2,\qquad Y=3\rho^2/4=3X^2,\qquad
 1+2X+3X^2-Y=1-\rho\ge0.
\]
Thus each point satisfies these attainability conditions.
The origin is already attainable by Proposition~\ref{prop:elementary},
so the entire segment belongs to $\Sigma_2$.

A nonreal point with $-1/2\le X<0$ and
$0<Y<1+2X+3X^2$, or with $-1<X\le-1/2$ and
$0<Y<Y_-(X)$, satisfies all four inequalities in
\eqref{eq:four-conditions} strictly. By continuity, the inequalities
remain valid in a small open neighborhood of the point.
Every point in that neighborhood is attainable, so the original point
is in the interior of $\Sigma_2$.
For a real point $-1<X<0$, the same inequalities at $Y=0$ are strict:
\[
 1-X^2>0,\qquad 1+2X+3X^2>0,\qquad
 G_X(0)=(X+1)^2(9X^2+2X+1)>0.
\]
By continuity, the inequalities remain valid for nearby real parts and
sufficiently small imaginary parts. Conjugation includes points below
the real axis, and Proposition~\ref{prop:elementary} includes the real
points themselves. Thus an open neighborhood of each point in $(-1,0)$
is contained in $\Sigma_2$.

Proposition~\ref{prop:elementary} places $[0,1]$, $[0,i]$ and $[0,-i]$
on the boundary: the adjacent sector $0<|\Arg\lambda|<\pi/2$ contains
no attainable eigenvalues. The point $-1$ is on the boundary because it
is attainable and $\Sigma_2$ lies in the unit disk.
Together with the two boundary curves and the attainability of all radii
$0<\rho\le r_2(x)$ for $1/4<x<1/2$ established above, these observations
prove the remaining assertions of Theorem~\ref{thm:main} for $q=2$.

\section*{Acknowledgment}
Vincent Ginis acknowledges support from the Research Foundation--Flanders (FWO) under grants No.~G032822N and G0K9322N.

\section*{Declaration of generative AI assistance}
ChatGPT~5.5 (OpenAI) and Fable~5.1 were used during preparation of the
original manuscript for brainstorming, organizing the exposition,
drafting, and revision. The present rewrite used GPT~6 through Codex
(OpenAI) for revision, literature comparison, and mathematical exposition.

\begingroup
\raggedright

\endgroup
\end{document}